\documentclass[preprint,12pt]{elsarticle}

\usepackage{amssymb}

\usepackage{graphicx}
\usepackage{amsmath}
\usepackage{hyperref}
\usepackage[textheight=20cm, textwidth=14cm]{geometry}

\usepackage{chngcntr}
\usepackage{apptools}
\usepackage{verbatim}
\AtAppendix{\counterwithin{theorem}{section}}
\AtAppendix{\counterwithin{proposition}{section}}

\newtheorem{theorem}{Theorem}
\newtheorem{proposition}{Proposition}
\newtheorem{lemma}{Lemma}
\newtheorem{corollary}{Corollary}

\newtheorem{problem}{Problem}

\newproof{proof}{Proof}

\newdefinition{remark}{Remark}

\numberwithin{equation}{section}

\newcommand\restr[2]{{
		\left.\kern-\nulldelimiterspace 
		#1 
		\vphantom{\big|} 
		\right|_{#2} 
}}

\journal{The Journal of Differential Equations}

\begin{document}

\begin{frontmatter}



\title{A non-local reduction principle for cocycles in Hilbert spaces\tnoteref{thanks}}


\author{Mikhail Anikushin}
\tnotetext[thanks]{This work is supported by V.~A. Rokhlin grant for young mathematicians of Saint Petersburg.}
\address{Department of
	Applied Cybernetics, Faculty of Mathematics and Mechanics,
	Saint-Petersburg State University, Universitetskiy prospekt 28, Peterhof, Saint-Petersburg, Russia, 198504}
\ead{demolishka@gmail.com}

\begin{abstract}
We study cocycles (non-autonomous dynamical systems) satisfying a certain squeezing condition with respect to the quadratic form of a bounded self-adjoint operator acting in a Hilbert space. We prove that (under additional assumptions) the orthogonal projector maps the fibres of some invariant set, containing bounded trajectories, in a one-to-one manner onto the negative subspace of the operator. This allows to reduce interesting dynamics onto this invariant set, which in some cases can be considered as a kind of inertial manifold for the cocycle. We consider applications of the reduction principle for periodic cocycles. For such cocycles we give an extension of the Massera second theorem, obtain the conditions for the existence of a Lyapunov stable periodic trajectory and prove convergence-type results, which we apply to study nonlinear periodic in time delayed-feedback equations posed in a proper Hilbert space and parabolic problems with a nonlinear periodic in time boundary control. The required operator is obtained as a solution to certain operator inequalities with the use of the Yakubovich-Likhtarnikov frequency theorem for $C_{0}$-semigroups and its properties are established from the Lyapunov inequality and dichotomy of the linear part of the problem.
\end{abstract}




\end{frontmatter}


\section{Introduction}
We start with the precise statements of our main results.

Suppose $\mathcal{Q}$ is a metric space and $\vartheta^{t} \colon \mathcal{Q} \to \mathcal{Q}$, $t \in \mathbb{R}$, is a dynamical system (=flow) on $\mathcal{Q}$. A \textit{cocycle} (in a separable Hilbert space $\mathbb{H}$) over the dynamical system $\left( \mathcal{Q}, \{ \vartheta^{t} \} \right)$ is a family of maps $\psi^{t} \colon \mathcal{Q} \times \mathbb{H}  \to \mathbb{H}$, $t \geq 0$, satisfying the following properties \cite{Cheban2020Book}:
\begin{enumerate}
	\item $\psi^{0}(q,u) = u$ for every $u \in \mathbb{H}, q \in \mathcal{Q}$.
	\item $\psi^{t+s}(q,u)=\psi^{t}(\vartheta^{s}(q),\psi^{s}(q,u))$ for all $u \in \mathbb{H}, q \in \mathcal{Q}$ and $t,s \geq 0$.
	\item The map $\mathbb{R}_{+} \times \mathcal{Q} \times \mathbb{H} \to \mathbb{H}$ defined as $(t,q,u) \mapsto \psi^{t}(q,u)$ is continuous.
\end{enumerate}

For the sake of brevity, we denote the cocycle by $(\psi, \vartheta)$. We will use the following main conditions imposed on the cocycle.

\begin{enumerate}
	\item[\textbf{(H1)}] There is a bounded linear operator $P \colon \mathbb{H} \to \mathbb{H}$, self-adjoint ($P=P^{*}$) and such that $\mathbb{H}$ splits into the direct sum of orthogonal $P$-invariant subspaces $\mathbb{H}^{+}$ and $\mathbb{H}^{-}$, i.~e. $\mathbb{H} = \mathbb{H}^{+} \oplus \mathbb{H}^{-}$, such that $\restr{P}{\mathbb{H}^{-}} < 0$ and $\restr{P}{\mathbb{H}^{+}}>0$.
	\item[\textbf{(H2)}] For some integer $j>0$ we have $\dim \mathbb{H}^{-} = j$.
	\item[\textbf{(H3)}] For $V(u):=(Pu,u)$ and some numbers $\delta>0$, $\nu>0$ we have
	\begin{multline}
	\label{EQ: SqueezingProperty}
	e^{2\nu r}V(\psi^{r}(q,u)-\psi^{r}(q,v))-e^{2\nu l} V(\psi^{l}(q,u)-\psi^{l}(q,v)) \leq \\ \leq -\delta\int_{l}^{r} e^{2\nu s} |\psi^{s}(q,u)-\psi^{s}(q,v)|^{2}ds,
	\end{multline}
	for every $u,v \in \mathbb{H}$, $q \in \mathcal{Q}$ and $0 \leq l \leq r$.
\end{enumerate}

\begin{remark}
	Suppose \textbf{(H1)} and \textbf{(H2)} are satisfied and $\mathbb{H}$ is real. Then the set $\mathcal{C}:=\{ u \in \mathbb{H} \ | \ (Pu,u) \leq 0 \}$ is called a $j$-\textit{dimensional quadratic cone}\footnote{The set $\mathcal{C}$ is a cone of rank $j$ in the terminology of \cite{FengWangWu2017}.} in $\mathbb{H}$. For every $u \in \mathbb{H}$ we write $u=(u^{+},u^{-})$, where $u^{+} \in \mathbb{H}^{+}$ and $u^{-} \in \mathbb{H}^{-}$. If $j=1$ and $e \in \mathbb{H}^{-}$ is non-zero we define $\mathcal{C}^{+}:=\{ u=(u^{+},u^{-}) \in \mathcal{C} \ | \ u^{-}=\alpha e, \alpha \geq 0 \}$. It can be shown that $\mathcal{C}^{+}$ is a closed convex cone in $\mathbb{H}$ and, therefore, it defines a partial order: $u \leqslant v$ iff $v-u \in \mathcal{C}^{+}$. In this case inequality \eqref{EQ: SqueezingProperty} implies that the cocycle is monotone (in the sense of the theory of monotone dynamical systems \cite{Cheban2019,Cheban2020Book}) w. r. t. the partial order given by $\mathcal{C}^{+}$. For $j>1$ there is no natural partial order due to the lack of convexity, but a pseudo-order can be defined and this also leads to some limitations for the dynamics \cite{FengWangWu2017}. However, when the quadratic cone $\mathcal{C}$ is obtained from a solution to certain operator inequalities (as in our case), the property in \textbf{(H3)} arises naturally. This lead to some topological conclusions (see below), which seem unreachable if we only consider the abstract monotonicity w.~r.~t. the pseudo-order as in \cite{FengWangWu2017}.
\end{remark}

A \textit{complete trajectory} is a continuous map $u \colon \mathbb{R} \to \mathbb{H}$ such that for some $q \in \mathcal{Q}$ the equality $u(t+s)=\psi^{t}(\vartheta^{s} (q),u(s))$ holds for all $t \geq 0$ and $s \in \mathbb{R}$. In this case we say that $u(\cdot)$ is \textit{passing through $u(0)$ at $q$}.

Within \textbf{(H3)} a complete trajectory $u(\cdot)$ of the cocycle is called \textit{amenable} if
\begin{equation}
\int\limits_{-\infty}^{0} e^{2\nu s}|u(s)|^{2} ds < \infty.
\end{equation}

Define $\mathfrak{A}(q)$ as the set of all $u_{0} \in \mathbb{H}$ such that there exists an amenable trajectory of the cocycle passing through $u_{0}$ at $q$. We call the set $\mathfrak{A}(q)$ the \textit{amenable set at $q$}. We consider the following assumption.

\begin{enumerate}
	\item[\textbf{(H4)}] For any $q \in \mathcal{Q}$ there exists at least one amenable trajectory at $q$, i.~e. the set $\mathfrak{A}(q)$ is non-empty.
\end{enumerate}

\noindent In applications, a bounded solution of some equation plays the role of an amenable trajectory required in \textbf{(H4)} (since $\nu>0$ all bounded trajectories are amenable).

We also make use of some compactness-like properties.

\begin{enumerate}
	\item[\textbf{(COM1)}] The operator $P$ from \textbf{(H1)} is compact.
\end{enumerate}

The assumption \textbf{(COM1)} makes the quadratic form of $P$ (that is $V(\cdot)$ from \textbf{(H3)}) continuous on bounded sets endowed with the weak topology of $\mathbb{H}$. It can be shown that within \textbf{(H1)} and \textbf{(H2)} the mentioned continuity of the quadratic form is equivalent to the compactness of $P$. A somewhat different compactness property is the following.

\begin{enumerate}
	\item[\textbf{(COM2)}] There exists $t_{com}>0$ such that the map $\psi^{t_{com}}(q,\cdot) \colon \mathbb{H} \to \mathbb{H}$ is compact for every $q \in \mathcal{Q}$.
\end{enumerate}

Compactness-like assumption \textbf{(COM2)} described in terms of the cocycle seems to be more convenient as it simplifies the proof of Theorem \ref{TH: SmithTheorem} and can be checked for cocycles generated by parabolic and delayed problems. Due to a smoothing effect in parabolic problems we may have both assumptions \textbf{(COM1)} and \textbf{(COM2)} satisfied (see \cite{Likhtarnikov1976, Lions1971} and Section \ref{SEC: Parabolic}). It seems that the compactness of $P$ does not hold in delayed problems (see \cite{Anikushin2019+OnCom}) in which the smoothing effect is delayed, but the corresponding evolution operator can still be compact for sufficiently large times (see \cite{Faheem1987,Webb1976} and Section \ref{SEC: DelayedSystems}). However, both assumptions lead to the compactness of bounded semi-trajectories (see Corollary \ref{COR: CompatTrajectories}), which we also use as a weaker hypothesis.
\begin{enumerate}
	\item[\textbf{(COM3)}] Any bounded semi-trajectory of the cocycle is compact.
\end{enumerate}

Consider the orthogonal projector $\Pi \colon \mathbb{H} \to \mathbb{H}^{-}$ onto $\mathbb{H}^{-}$. The following theorem is a generalization of a result of R.~A.~Smith (see Theorem 8 in \cite{Smith1986} or Theorem 5 in \cite{Smith1994}) and it is one of main results of the present paper.
\begin{theorem}
	\label{TH: SmithTheorem}
	Suppose that \textbf{(H1)}-\textbf{(H4)} and one of \textbf{(COM1)} or \textbf{(COM2)} hold. Then the map $\Pi_{q} := \restr{\Pi}{\mathfrak{A}(q)} \colon \mathfrak{A}(q) \to \mathbb{H}^{-}$ is a homeomorphism.
\end{theorem}

Without \textbf{(COM1)} or \textbf{(COM2)} we only know that the map $\Pi_{q}$ provides a homeomorphism onto a subset of $\mathbb{H}^{-}$ (Lemma \ref{LEM: PiIsHomeo} below). In the finite-dimensional case the map $\Pi_{q}$ is a bi-Lipschitz homeomorphism between $\mathfrak{A}(q)$ and $\mathbb{H}^{-}$ \cite{Smith1986}. This is due to the fact that the positive operator $\restr{P}{\mathbb{H}^{+}}$ generates an equivalent scalar product on $\mathbb{H}_{+}$. In the case of infinite dimensions the equivalence does not hold (that can be easily seen if $P$ is compact). 

In certain cases the family of sets $\mathfrak{A}(q)$, $q \in \mathcal{Q}$, can be considered as a kind of inertial manifold for the cocycle (see Section \ref{SEC: TheMapPhi}). See also Section 6 of \cite{Smith1994} for a comparison with the inertial manifolds theory.

The conclusion of Theorem \ref{TH: SmithTheorem} (or its weak version, Lemma \ref{LEM: PiIsHomeo}) is convenient if we are looking for an embedding of invariant sets into some finite-dimensional space. It allows, for example, to sharp and extend the estimates of M.~L.~Cartwright \cite{Cartwright1969} of additional frequencies in the Fourier spectrum of almost periodic solutions (see \cite{KuzLeoReit2019}, where this is done for almost periodic ODEs). Topological structure of amenable sets $\mathfrak{A}(q)$ stated in Theorem \ref{TH: SmithTheorem} can be used to extend delicate theorems of V.~V.~Zhikov for almost periodic ODEs in low dimensions (see Chapter VII in \cite{LevitanZhikov1982}) to high-dimensional and infinite-dimensional cases. We will deal with these generalizations somewhere else. In the case of semi-flows, Theorem \ref{TH: SmithTheorem} (with $j=2$) can be used to prove an analog of the Poincar\'{e}-Bendixson theorem and obtain conditions for the existence of an orbitally stable periodic orbit (see \cite{Anik2020PB}).

In this paper we concentrate on applications of the reduction principle for studying abstract periodic cocycles\footnote{A cocycle $(\psi,\vartheta)$ is called $\sigma$-periodic if $\mathcal{Q}=\mathcal{S}^{1}_{\sigma}:=\mathbb{R}/\sigma\mathbb{Z}$ and $\vartheta^{t}(\theta):=\theta+t$ for $\theta \in \mathcal{S}^{1}_{\sigma}$ and $t \in \mathbb{R}$. See Section \ref{SEC: AbstractPerCoc}.}. In this direction we generalize all the main results from \cite{Smith1986,Smith1990} (for periodic ODEs and RFDEs) to our abstract context. Comparing with the proofs of similar theorems in \cite{Smith1986,Smith1990}, our proofs are more topological (the mentonied results of R.~A. Smith rely on some exponential estimates and topological arguments are mixed with them). Namely, we obtain the following convergence theorem.

\begin{theorem}
	\label{TH: ConvergenceTh}
	Suppose that the cocycle $(\psi,\vartheta)$ is $\sigma$-periodic and let \textbf{(H1)}, \textbf{(H2)} with $j=1$, \textbf{(H3)} and \textbf{(COM3)} hold. Then for any bounded semi-trajectory $u(t):=\psi^{t}(q,u_{0})$, where $t \geq 0$, there exists a $\sigma$-periodic trajectory $u^{*}$ at $q$ such that $u(t)-u^{*}(t) \to 0$ as $t \to +\infty$.
\end{theorem}

Below we present some applications of Theorem \ref{TH: ConvergenceTh} for periodic cocycles generated by certain delay and parabolic equations. Note that results of convergence even in the case of almost periodic cocycles require (in addition to some low-dimensionality of fibres) certain stability assumptions (see \cite{Cheban2019,Cheban2020Book} and references therein).

A bounded closed subset $\mathcal{S}_{0} \subset \mathbb{H}$ is called a \textit{sink} for the cocycle $(\psi,\vartheta)$ if there exists an open set $\mathcal{G} \supset \mathcal{S}_{0}$ such that for every $u_{0} \in \mathcal{G}$ and $q \in \mathcal{Q}$ there exists $T = T(q, u_{0}) \geq 0$ such that $\psi^{t}(q,u_{0}) \in \mathcal{S}_{0}$ for all $t \geq T$.

\begin{theorem}
	\label{TH: StablePeriodicSmith}
	Suppose that the cocycle $(\psi,\vartheta)$ is $\sigma$-periodic and let \textbf{(H1)}, \textbf{(H2)} with $j=1$, \textbf{(H3)} and one of \textbf{(COM1)} or \textbf{(COM2)} hold. If a bounded closed set $\mathcal{S}_{0}$ is a \textit{sink} for the cocycle $(\psi,\vartheta)$ then there exists a Lyapunov stable $\sigma$-periodic trajectory in $\mathcal{S}_{0}$.
\end{theorem}

Under the conditions of Theorem \ref{TH: ConvergenceTh} any isolated Lyapunov stable $\sigma$-periodic trajectory is asymptotically stable (see Proposition \ref{PROS: IsolatedStable}).

\begin{remark}
	In applications, the property in \textbf{(H3)} is linked with the Lipschitzity of nonlinearities. Often these conditions are not satisfied globally and holds only on a certain invariant set $\mathcal{S}$. Redefining the nonlinearities outside $\mathcal{S}$ to make them well-behaved globally, one can still apply Theorems \ref{TH: ConvergenceTh} and \ref{TH: StablePeriodicSmith} for the modified system to derive results for the original one, but only on the set $\mathcal{S}$. Non-trivial applications of this are given in \cite{Anik2020PB}.
\end{remark}

We also derive an extension of the Massera second theorem \cite{Massera1950} as follows.
\begin{theorem}
	\label{TH: MasseraTheorem}
	Suppose that the cocycle $(\psi,\vartheta)$ is $\sigma$-periodic and let \textbf{(H1)}, \textbf{(H2)} with $j=2$, \textbf{(H3)} and one of \textbf{(COM1)} or \textbf{(COM2)} hold. If there exists a bounded semi-trajectory $u(t)=\psi^{t}(q,u_{0})$, where $t \geq 0$, at $q \in \mathcal{Q}$ then there exists a $\sigma$-periodic trajectory.
\end{theorem}

For ODEs, various analogs of \textbf{(H3)} are well-known and were widely used to study stability, existence of forced oscillations and dimension-like properties \cite{LeoBurShep1996,Smith1986,Anikushin2019ND,AnikushinRR2019,Anikushin2019Vestnik, KuzLeoReit2019}. The key point here is that \textbf{(H3)} can be effectively verified with the use of the Yakubovich-Kalman frequency theorem \cite{Gelig1978}. Infinite-dimensional versions of the frequency theorem are known \cite{Likhtarnikov1976,Likhtarnikov1977,LouisWexler1991}, but their applications are usually considered for the problems of absolute stability and optimal control\footnote{However, even in these directions applications encounter a couple of obstacles.}, and rarely seen in the direction of the present paper \cite{KalininReitmann2012,Popov2014}. 

In order to apply our results to delay equations we consider them in the $L_{2}$-setting. For checking the well-posedness of delay equations the theory presented in \cite{Webb1981,Faheem1987,BatkaiPiazzera2005,Webb1976} is useful. 

For our purposes to get the operator $P$ from \textbf{(H3)} the frequency theorem is applied to a linear equation with an unstable operator or more precisely, to the pair $(A,B)$, where $A$ denotes the linear part and $B$ is a control operator (for example, boundary operator). For these applications it is required the $L_{2}$-controllability of the pair $(A,B)$ or, that turns out to be equivalent, its exponential stabilizability (see Appendix \ref{SEC: FrequencyTheorem}). In concrete examples this property can be checked by direct calculations (see Sections \ref{SEC: DelayedSystems} and \ref{SEC: Parabolic} and also \cite{Anik2020PB}). In most cases the stabilizability can be considered as a non-degeneracy condition for the pair $(A,B)$. However, the problem of boundary stabilizability for parabolic problems is less trivial and was considered in several papers (see, for example, \cite{Lasiecka1983,Liu2003} and references therein). This condition makes it harder to compare\footnote{One of the main conditions in \cite{Smith1990,Smith1994} is of frequency-domain type and as it shown in \cite{Anik2020PB} it implies the usual frequency condition, i.~e. the condition $\alpha_{3}<0$, where $\alpha_{3}$ is from Appendix \ref{SEC: FrequencyTheorem}.} our results with the corresponding ones in \cite{Smith1986,Smith1990}. The use of the frequency theorem allows us to study problems, which were not considered by R.~A.~Smith (such as parabolic problems with boundary controls), and sometimes leads to more sharper estimates than those of R.~A.~Smith (see \cite{Anik2020PB} for a comparison).

So, the presented approach (along with \cite{Anik2020PB}) is an attempt to unify the results of \cite{Smith1986,Smith1990,Smith1994} and other papers of R.~A.~Smith. While the exponential stabilizability does not seem so restrictive (but sometimes hard to verify), we cannot apply our theory to study delay equations with an unbounded (in $L_{2}$) measurement operator due to the limitations in the current versions of the frequency theorem (see Remarks \ref{REM: DelayLimitations} and \ref{REM: LimitFreqCond}). We hope that our results will also stimulate developments of the frequency-domain methods.

This paper is organized as follows. In Section \ref{SEC: AmenableSets} we give a proof of Theorem \ref{TH: SmithTheorem}. In Section \ref{SEC: TheMapPhi} we study properties of the family of sets $\mathfrak{A}(q)$, $q \in \mathcal{Q}$, which in certain cases give rise to an invariant w.~r.~t. the skew-product flow topological manifold, which attracts compact semi-trajectories. In Section \ref{SEC: AbstractPerCoc} we prove Theorems \ref{TH: ConvergenceTh}, \ref{TH: StablePeriodicSmith} and \ref{TH: MasseraTheorem}. In Section \ref{SEC: DelayedSystems} we consider delay equations modelling analog neural networks posed in a proper Hilbert space. For a single scalar equation we present a complete analysis of all the conditions necessary to apply Theorem \ref{TH: ConvergenceTh}. In Section \ref{SEC: Parabolic} by similar techniques we study the parabolic problem of heating of a one-dimensional rode under a monotone nonlinear boundary control. In Appendix \ref{SEC: FrequencyTheorem}, for the convenience of the reader, we expound the frequency theorem of Yakubovich-Likhtarnikov for $C_{0}$-semigroups. In Appendix \ref{SEC: LyapunivIneq} we collect (with proofs) some propositions that are useful in studying spectral properties of the operator $P$.
\section{Structure of amenable sets}
\label{SEC: AmenableSets}
We say that a continuous map $u \colon [-T;+\infty) \to \mathbb{H}$ is a \textit{trajectory of the cocycle at $q$ defined for $t \geq -T$} if $u(t+s)=\psi^{t}(\vartheta^{s}(q),u(s))$ for every $t \geq 0$ and $s \in [-T,+\infty)$. It is easy too see that in terms of trajectories $u(\cdot)$ and $v(\cdot)$ (both at $q$ and defined for $t \geq -T$) inequality \eqref{EQ: SqueezingProperty} can be written as
\begin{equation}
\label{EQ: H3TrajectoryForm}
e^{2\nu r} V(u(r)-v(r)) - e^{2\nu l}V(u(l)-v(l)) \leq -\delta\int_{l}^{r} e^{2\nu s} |u(s) - v(s)|^{2}ds
\end{equation}
for every $-T \leq l\leq r < \infty$. We will usually refer to \textbf{(H3)} in such a form.

\begin{lemma}
	\label{LEM: AmenableLemma}
	Within \textbf{(H1)} and \textbf{(H3)} let $u^{*}$ and $v^{*}$ be complete trajectories at $q \in \mathcal{Q}$. We have
	\begin{enumerate}
		\item[1)] If both $u^{*}$ and $v^{*}$ are amenable then $V(u^{*}(t)-v^{*}(t)) \leq 0$ for every $t \in \mathbb{R}$.
		\item[2)] If $u^{*}$ is amenable and $V(u^{*}(t)-v^{*}(t)) \leq 0$ for every $t \in \mathbb{R}$ then $v^{*}$ is amenable.
	\end{enumerate}
\end{lemma}
\begin{proof}
	1) Using \textbf{(H3)} we consider \eqref{EQ: H3TrajectoryForm} for $u^{*}(\cdot)$ and $v^{*}(\cdot)$ at $[l,t]$:
	\begin{equation}
	\label{EQ: AmenablePropGeneral}
	e^{2\nu t} V(u^{*}(t)-v^{*}(t)) - e^{2 \nu l}V(u^{*}(l)-v^{*}(l)) \leq -\delta \int\limits_{l}^{t}e^{2\nu s}|u^{*}(s)-v^{*}(s)|^{2}ds.
	\end{equation}
	Since $u^{*}$ and $v^{*}$ are amenable the integral $\int_{-\infty}^{0} e^{2\nu s} |u^{*}(s)-v^{*}(s)|^{2} ds$ converges and as a consequence there is a sequence of $l=l_{k} \to -\infty$ such that
	$e^{2\nu l_{k}} V(u^{*}(l_{k}) - v^{*}(l_{k})) \to 0$ as $k \to \infty$. Considering \eqref{EQ: AmenablePropGeneral} with $l=l_{k}$ and taking it to the limit as $k \to \infty$ we get
	\begin{equation}
	e^{2\nu t} V(u^{*}(t)-v^{*}(t)) \leq -\delta \int\limits_{-\infty}^{t}e^{2\nu s}|u^{*}(s)-v^{*}(s)|^{2}ds
	\end{equation}
	that proves $V(u^{*}(t)-v^{*}(t)) \leq 0$ for all $t \in \mathbb{R}$.
	
	2) Consider \eqref{EQ: AmenablePropGeneral} with $t=0$ and $l \to -\infty$. From the property $V(u^{*}(l)-v^{*}(l)) \leq 0$ for all $l \in \mathbb{R}$ we get
	\begin{equation}
	\label{EQ: Lemma1SmithAmenable}
	-V(u^{*}(0)-v^{*}(0)) \geq \delta \int\limits_{-\infty}^{0}e^{2\nu s}|u^{*}(s)-v^{*}(s)|^{2}ds.
	\end{equation}
	Thus the amenability of $v^{*}$ follows from \eqref{EQ: Lemma1SmithAmenable} and the Minkowski inequality.
\end{proof}

Within \textbf{(H1)} any $u \in \mathbb{H}$ can be represented uniquely as $u=u^{+} + u^{-}$, where $u^{+} \in \mathbb{H}^{+}$ and $u^{-} \in \mathbb{H}^{-}$. It is clear that $V(u)=(Pu,u) = (\restr{P}{\mathbb{H}^{+}}u^{+},u^{+}) + (\restr{P}{\mathbb{H}^{-}}u^{-},u^{-})$.

\begin{lemma}
	\label{LEM: PiIsHomeo}
	Let \textbf{(H1)}, \textbf{(H3)}, \textbf{(H4)} hold and $q$ be fixed; then the map $\Pi_{q} = \restr{\Pi}{\mathfrak{A}(q)} \colon \mathfrak{A}(q) \to \Pi \mathfrak{A}(q)$ is a homeomorphism.
\end{lemma}
\begin{proof}
	Clearly, $\Pi$ is continuous. Suppose $u^{*}(\cdot)$ and $v^{*}(\cdot)$ are two amenable trajectories passing trough $u^{*}(0)$ and $v^{*}(0)$ respectively at $q$. From \textbf{(H3)}, \textbf{(H1)}, the amenability and the Cauchy–Bunyakovsky–Schwarz inequality we have
	\begin{equation}
	\label{EQ: AmenablePiHomeo}
	\left| \restr{P}{\mathbb{H}^{-}} \right| \cdot \left| \Pi u^{*}(0) - \Pi v^{*}(0) \right|^{2} \geq -V(u^{*}(0)-v^{*}(0)) \geq \delta \int_{-\infty}^{0}e^{2\nu s} | u^{*}(s) - v^{*}(s)|^{2}ds.
	\end{equation}
	From \eqref{EQ: AmenablePiHomeo} it is clear that $\Pi$ is injective on $\mathfrak{A}(q)$. 
	
	Suppose $\Pi v^{(0)}_{k}$, $v^{(0)}_{k} \in \mathfrak{A}(q)$ for $k=1,2,\ldots$, converges to $\Pi u^{*}_{0}$, $u^{*}_{0} \in \mathfrak{A}(q)$. Consider the corresponding amenable trajectories $v^{*}_{k}(\cdot)$ and $u^{*}(\cdot)$ at $q$ passing through $v^{(0)}_{k}=v^{*}_{k}(0)$ and $u_{0}=u^{*}(0)$ respectively. Suppose that $v^{*}_{k}(0)$ does not converge to $u^{*}_{0}$. Then there exists a subsequence $v^{*}_{k_{m}}(0)$, $m=1,2,\ldots$, and $\delta_{0}>0$ such that $|v^{*}_{k_{m}}(0) - u^{*}(0) | \geq \delta_{0} > 0$. From a similar to \eqref{EQ: AmenablePiHomeo} inequality we have that
	\begin{equation}
	\int_{-\infty}^{0}e^{2\nu s} | u^{*}(s) - v^{*}_{k_{m}}(s)|^{2} \to 0.
	\end{equation}
	In particular, by the mean value theorem, for some sequence $t_{m} \in [-2,-1]$ we have that $v^{*}_{k_{m}}(t_{m}) - u^{*}(t_{m}) \to 0$ as $m \to \infty$. We may assume that $t_{m}$ converges to some $\overline{t} \in [-2,-1]$. From this it follows that $v^{*}_{k_{m}}(t_{m}) \to u^{*}(\overline{t})$ Using the continuity of the cocycle we get the convergence $v^{*}_{k_{m}}(t) \to u^{*}(t)$ for $t \in (\overline{t},+\infty)$ and, in particular, $v^{*}_{k_{m}}(0) \to u^{*}(0)$ that leads to a contradiction.
\end{proof}

\begin{corollary}
	\label{COR: UniqueAmenable}
	Let the conditions of Lemma \ref{LEM: PiIsHomeo} be satisfied. Then for every $u_{0} \in \mathfrak{A}_{q}$ there exists a unique amenable trajectory passing through $u_{0}$ at $q$.
\end{corollary}
\begin{proof}
	This is an immediate consequence of \eqref{EQ: AmenablePiHomeo}.
\end{proof}

Let $q \in \mathcal{Q}$ and $v \in \mathbb{H}$ be fixed. For two real numbers $T_{1} < T_{2}$ consider the map $G_{T_{1}}^{T_{2}} \colon \mathbb{H}^{-} \to \mathbb{H}^{-}$ defined as
\begin{equation}
\label{EQ: OperatorG}
G_{T_{1}}^{T_{2}}(\zeta) = G_{T_{1}}^{T_{2}}(\zeta; q, v) := \Pi \psi^{T_{2}-T_{1}}(\vartheta^{T_{1}}(q), v + \zeta).
\end{equation}

\begin{lemma}
	\label{LEM: OperatorG}
	Assume \textbf{(H1)}, \textbf{(H2)} and \textbf{(H3)} are satisfied; then $G_{T_{1}}^{T_{2}} \colon \mathbb{H}^{-} \to \mathbb{H}^{-}$ defined in \eqref{EQ: OperatorG} is a homeomorphism.
\end{lemma}
\begin{proof}
	Let $\zeta_{1}, \zeta_{2} \in \mathbb{H}^{-}$ and consider two corresponding trajectories $u_{i}(t):=\psi^{t - T_{1}}(\vartheta^{T_{1}}(q), v + \zeta_{i})$, where $i=1,2$ and $t \geq T_{1}$. From \textbf{(H3)} we have
	\begin{equation}
	\label{EQ: OperatorGFromH3}
	-e^{2\nu T_{2}} V( u_{1}(T_{2})-u_{2}(T_{2}) ) + e^{2\nu T_{1}} V(u_{1}(T_{1})-u_{2}(T_{1})) \geq \delta \int\limits_{T_{1}}^{T_{2}} e^{2\nu s} |u_{1}(s)-u_{2}(s)|^{2}ds.
	\end{equation}
	Note that $u_{1}(T_{1}) - u_{2}(T_{1}) = \zeta_{1}-\zeta_{2} \in \mathbb{H}^{-}$ and therefore, by \textbf{(H1)}, $V(u_{1}(T_{1})-u_{1}(T_{1})) \leq 0$. From this, \eqref{EQ: OperatorGFromH3} and the inequality
	\begin{equation*}
	-V(u_{1}(T_{2})-u_{1}(T_2)) \leq \left| \restr{P}{\mathbb{H}^{-}} \right| \cdot \left| \Pi u_{1}(T_{2}) - \Pi u_{2}(T_{2}) \right|
	\end{equation*}
	we get
	\begin{equation}
	\label{EQ: OperatorGKeyIneq}
	\left|\Pi u_{1}(T_{2}) - \Pi u_{2} (T_{2}) \right|^{2} \geq \delta \left|\restr{P}{\mathbb{H}^{-}}\right|^{-1} e^{-\nu T_{2}} \int\limits_{T_{1}}^{T_{2}} e^{2\nu s} |u_{1}(s)-u_{2}(s)|^{2}ds.
	\end{equation}
	 By definition, $G_{T_{1}}^{T_{2}}(\zeta_{1}) - G_{T_{1}}^{T_{2}}(\zeta_{2}) = \Pi u_{1}(T_{2})- \Pi u_{2}(T_{2})$ that is contained in the left-hand side of \eqref{EQ: OperatorGKeyIneq}. From \eqref{EQ: OperatorG} and since the cocycle is continuous the map $G_{T_{1}}^{T_{2}}$ is continuous too. From \eqref{EQ: OperatorGKeyIneq} the injectivity of $G_{T_{1}}^{T_{2}}$ follows at once. By the Brouwer theorem on invariance of domain, $G_{T_{1}}^{T_{2}}(\mathbb{H}^{-})$ is open in $\mathbb{H}^{-}$ and $G_{T_{1}}^{T_{2}}$ realizes a homeomorphism between $\mathbb{H}^{-}$ and its image. We will show that the image $G_{T_{1}}^{T_{2}}(\mathbb{H}^{-})$ is closed.
	 Suppose we are given with a fundamental sequence $\Pi u_{k}(T_{2})$, where $k=1,2,\ldots$, corresponding to $u_{k}(t)=\psi^{t-T_{1}}(\vartheta^{T_{1}}(q), \zeta_{k}+v)$. From \textbf{(H3)} for some $\delta_{0}>0$ we have
	 \begin{equation}
	 \label{EQ: OperatorGClosedness}
	 \begin{split}
	 &e^{2\nu T_{2}} \left| \restr{P}{\mathbb{H}^{-}} \right| \cdot \left| \Pi u_{k}(T_{2}) - \Pi u_{m}(T_{2}) \right| \geq -e^{2\nu T_{2}}V(u_{k}(T_{2})-u_{m}(T_{2})) \geq \\
	 &-e^{2\nu T_{1}}V(u_{k}(T_{1})-u_{m}(T_{1})) \geq -e^{2\nu T_{1}} \cdot \delta_{0} \cdot |\zeta_{k}-\zeta_{m}|^{2}.
	 \end{split}
	 \end{equation}
	 By \eqref{EQ: OperatorGClosedness} the sequence $\{ \zeta_{k} \}$ is also fundamental. We denote its limit by $\overline{\zeta}$ and, by continuity of $G_{T_{1}}^{T_{2}}$ we have $G_{T_{1}}^{T_{2}}(\zeta_{k})=\Pi u_{k}(T_{2}) \to G_{T_{1}}^{T_{2}}(\overline{\zeta})$ as $k \to \infty$ that proves the closedness. Thus $G_{T_{1}}^{T_{2}}(\mathbb{H}^{-})$ is open and closed in $\mathbb{H}^{-}$ and, consequently, $G_{T_{1}}^{T_{2}}(\mathbb{H}^{-})=\mathbb{H}^{-}$.
\end{proof}

\begin{corollary}
	\label{COR: OperatorGCont}
	Within the assumptions of Lemma \ref{LEM: OperatorG} consider the set $T_{1,2}:=\{ (T_{1}, T_{2}) \in \mathbb{R}^{2} \ | \ T_{1} < T_{2} \}$; then the map $G \colon T_{1,2} \times \mathbb{H}^{-} \to T_{1,2} \times \mathbb{H}^{-}$ defined as
	\begin{equation}
	G(T_{1},T_{2}, \zeta) := (T_{1},T_{2},G_{T_{1}}^{T_{2}}(\zeta)), 
	\end{equation}
	where $G_{T_{1}}^{T_{2}}(\zeta):=G_{T_{1}}^{T_{2}}(\zeta; q, v)$, is a homeomorphism.
\end{corollary}
\begin{proof}
	From \eqref{EQ: OperatorG} it is clear that $G$ is continuous. By Lemma \ref{LEM: OperatorG}, $G$ is bijective and since $T_{1,2}$ is homeomorphic to $\mathbb{R}^{2}$ we may apply the theorem on invariance of domain again.
\end{proof}
\begin{remark}
	\label{REM: RemarkBrowerSecond}
	From Corollary \ref{COR: OperatorGCont} it follows that if $T^{(k)}_{1} \to T_{1}$, $T_{2}^{(k)} \to T_{2}$, where $T_{1}<T_{2}$, and $\zeta_{k} \to \zeta$ in $\mathbb{H}^{-}$ then $\left( G_{T^{(k)}_{1}}^{T_{2}^{(k)}} \right)^{-1}(\zeta_{k}) \to \left( G_{T_{1}}^{T_{2}} \right)^{-1}(\zeta)$.
\end{remark}

\begin{lemma}
	\label{LEM: GeneralCompactnessLemma}
	Let $v_{k}(\cdot)$, where $k=1,2,\ldots$, be a sequence of trajectories at $q \in \mathcal{Q}$, which are defined for $t \geq T_{k}$ with $T_{k} \to -\infty$ as $k \to +\infty$. Suppose also that for every $k$ there is a sequence $t^{(l)}_{k}$, where $l=-1,-2,\ldots$, satisfying the properties
	\begin{enumerate}
		\item[1)] For every $l$ there is a compact interval $I_{l}=[a_{l},b_{l}]$ such that $t^{(l)}_{k} \in I_{l}$ for all $k=1,2,\ldots$. Moreover, $b_{l} \to -\infty$ as $l \to -\infty$.
		\item[2)] For every $l$ the sequence $v_{k}(t^{(l)}_{k})$ is bounded uniformly for all sufficiently large $k$ (depending on $l$).
	\end{enumerate}
	Suppose that either \textbf{(COM2)} or \textbf{(H1)}, \textbf{(H2)}, \textbf{(H3)} and \textbf{(COM1)} are satisfied. Then there exists a complete trajectory $v^{*}(\cdot)$ and a subsequence $v_{k_{m}}(\cdot)$, where $m=1,2,\ldots$, such that $v_{k_{m}}(t) \to v^{*}(t)$ as $m \to +\infty$ for every $t \in \mathbb{R}$.
\end{lemma}
\begin{proof}
	By the property in item 1) we may suppose that $b_{l-1} < a_{l}$ for all $l=-1,-2,\ldots$.
	
	Case of \textbf{(COM2)}: Since the sequence $t^{(l)}_{k}$ is bounded in $k$, there is a subsequence $t^{(l)}_{k_{m}} \to \overline{t}_{l}$ as $m \to +\infty$ for some $\overline{t}_{l} \in I_{l}$. Note that $\overline{t}_{l-1} < \overline{t}_{l}$. Let us consider $l$ such that $\overline{t}_{l} + t_{com} < 0$. By item 2) the sequence $v_{k_{m}}(t^{(l)}_{k_{m}})$ is bounded (uniformly in $m$) in $\mathbb{H}$, and therefore the sequence
	\begin{equation}
    v_{k_{m}}(t^{(l)}_{k_{m}}+t_{com})=\psi^{t_{com}}(\vartheta^{t^{(l)}_{k_{m}}}(q), v_{k_{m}}(t^{(l)}_{k_{m}}))
	\end{equation}  
	lies in a compact set of $\mathbb{H}$ due to \textbf{(COM2)}. Therefore there is a subsequence (for the sake of brevity we keep the same index) of $v_{k_{m}}(t^{(l)}_{k_{m}}+t_{com})$ converging strongly to some $\overline{v}_{l}$ as $m \to \infty$. Using Cantor's diagonal procedure we may assume that the latter convergence holds for all $l$. Now consider the trajectories
	\begin{equation}
	\overline{v}_{l}(t):=\psi^{t-(\overline{t}_{l}+t_{com})}(\vartheta^{\overline{t}_{l}+t_{com}}(q), \overline{v}_{l}), \text{ defined for } t \geq \overline{t}_{l}+t_{com}.
	\end{equation}
	Since $v_{k_{m}}(t^{(l)}_{k_{m}}+t_{com}) \to \overline{v}_{l}$ we have the convergence $v_{k_{m}}(t) \to \overline{v}_{l}(t)$ for all $t \in (\overline{t}_{l}+t_{com},+\infty)$. From this it follows that $v_{l}(t)$ and $v_{l-1}(t)$ coincide for all $t \in (\overline{t}_{l}+t_{com},+\infty)$. Thus, the equality
	\begin{equation}
	v^{*}(t):=\overline{v}_{l}(t), \text{ for arbitrary } l \text{ such that } t > \overline{t}_{l}
	\end{equation}
	correctly defines a complete trajectory of the cocycle such that $v_{k_{m}}(t) \to v^{*}(t)$ for every $t \in (-\infty,+\infty)$.
	
	Case of \textbf{(COM1)}: Using Cantor's diagonal procedure one can obtain a subsequence $v_{k_{m}}$, $m=1,2,\ldots$, such that $v_{k_{m}}(t^{(l)}_{k_{m}})$ is well-defined for sufficiently large $m$ (depending on $l$) and converges weakly to some $\overline{v}_{l} \in \mathbb{H}$ and $t^{(l)}_{k_{m}} \to \overline{t}_{l} \in I_{l}$ as $m \to \infty$. Now our purpose is to show that this convergence holds in the strong topology. Assume that for some $l$ the strong convergence does not hold. Then for some subsequence (we keep the same index) we have $|v_{k_{m}} (t^{(l)}_{k_{m}}) - \overline{v}_{l}| \geq \delta_{0} > 0$ for all $m=1,2,\ldots$. We consider the trajectories $\varphi_{m}$, where 
	\begin{equation}
	\label{EQ: SmithThDefiphim}
	\varphi_{m}(t) := \psi^{t-t^{(l-1)}_{k_{m}}}\left(\vartheta^{t^{(l-1)}_{k_{m}}}(q), \left( G_{t^{(l-1)}_{k_{m}}}^{t^{(l)}_{k_{m}}} \right)^{-1}\left( \Pi \overline{v}_{l} ; q, \overline{v}_{l-1} \right) \right)
	\end{equation}
	for $t \geq t_{k_{m}}^{(l-1)}$ and the trajectory $\varphi$ (defined for $t \geq \overline{t}_{l-1}$)
	\begin{equation}
	\label{EQ: SmithThDefiphi}
	\varphi(t):=\psi^{t - \overline{t}_{l-1}} \left( \vartheta^{\overline{t}_{l-1}} (q), \left(G_{\overline{t}_{l-1}}^{\overline{t}_{l}}\right)^{-1}( \Pi \overline{v}_{l}; q, \overline{v}_{l-1}) \right). 
	\end{equation}
	From Corollary \ref{COR: OperatorGCont} (see Remark \ref{REM: RemarkBrowerSecond}) and since the cocycle is continuous it follows that $\varphi_{m}( t ) \to \varphi( t )$ in $\mathbb{H}$ for every $t \in (\overline{t}_{l-1}, +\infty)$. Moreover, from Corollary \ref{COR: OperatorGCont} it follows that $\varphi_{m}( t^{(l-1)}_{k_{m}} ) \to \varphi(\overline{t}_{l-1})$. From \textbf{(H3)} for a small number $\varepsilon>0$ we have
	\begin{multline}
	\label{EQ: IntegralIneqSmithTh}
	-e^{2\nu t^{(l)}_{k_{m}}}V\left( \varphi_{m}\left( t^{(l)}_{k_{m}} \right)-v_{k_{m}} \left( t^{(l)}_{k_{m}} \right) \right) + e^{2\nu t^{(l-1)}_{k_{m}}} V\left(\varphi_{m}\left( t^{(l-1)}_{k_{m}} \right) - v_{k_{m}}\left( t^{(l-1)}_{k_{m}} \right) \right) \geq \\ \geq \delta \cdot \int\limits_{\overline{t}_{l-1} +\varepsilon}^{\overline{t}_{l} - \varepsilon} e^{2\nu s}|\varphi_{m}(s)-v_{k_{m}}(s)|^{2}ds.
	\end{multline}
	Now we will deal with the two terms in the left-hand side of \eqref{EQ: IntegralIneqSmithTh}. From \eqref{EQ: SmithThDefiphim} we have $\Pi \varphi_{m}\left( t^{(l)}_{k_{m}}\right) = \Pi \overline{v}_{l}$ and therefore by \textbf{(H1)} and Cauchy-Bunyakovsky-Schwarz inequality
	\begin{equation}
	\label{EQ: SmithThToDealWithInt1}
	-V\left( \varphi_{m}\left( t^{(l)}_{k_{m}} \right)-v_{k_{m}} \left( t^{(l)}_{k_{m}} \right) \right) \leq \left|\restr{P}{\mathbb{H}^{-}}\right| \cdot \left|\Pi \overline{v}_{l} - \Pi v_{k_{m}} \left( t^{(l)}_{k_{m}} \right)\right|^{2}.
	\end{equation}
	Since $\Pi$ has finite-dimensional range the sequence $\Pi v_{k_{m}} \left( t^{(l)}_{k_{m}} \right)$ converges strongly to $\Pi \overline{v}_{l}$ as $m \to \infty$ and thus the right-hand side in \eqref{EQ: SmithThToDealWithInt1} tends to zero. The sequence $\varphi_{m}\left( t^{(l-1)}_{k_{m}} \right)$ converges to $\varphi(\overline{t}_{l-1})$ in $\mathbb{H}$. From \eqref{EQ: SmithThDefiphi} we have the property $V(\varphi(\overline{t}_{l-1})-\overline{v}_{l-1}) \leq 0$ and the only use of \textbf{(COM1)} is the following
	\begin{equation}
	\label{EQ: OnlyComUse}
	V\left( \varphi_{m}\left( t^{(l-1)}_{k_{m}} \right) - v_{k_{m}} \left( t^{(l-1)}_{k_{m}} \right) \right) \to V(\varphi(\overline{t}_{l-1}) - \overline{v}_{l-1}) \leq 0.
	\end{equation}
	Thus, the left-hand side of \eqref{EQ: IntegralIneqSmithTh} tends to zero as $m \to \infty$. It follows that there exists a subsequence (we keep the same index) such that $\varphi_{m}(s) - v_{k_{m}}(s) \to 0$ in $\mathbb{H}$ for almost all $s \in [\overline{t}_{l-1}+\varepsilon,\overline{t}_{l}-\varepsilon]$. But since $\varphi_{m}(s) \to \varphi(s)$ it is necessary that $\varphi(s) - v_{k_{m}}(s) \to 0$ for almost all $s \in [\overline{t}_{l-1}+\varepsilon,\overline{t}_{l} - \varepsilon]$. In particular, the strong convergence holds for some $\overline{t}_{l-1} < s_{0} < \overline{t}_{l}$ and this implies (since the cocycle is continuous) the strong convergence of $v_{k_{m}}(s)$ to $\varphi(s)$ uniformly in $s$ from compact subsets of $[s_{0},+\infty)$. In particular, from
	\begin{equation}
	\left| \varphi(\overline{t}_{l}) - v_{k_{m}}\left( t^{(l)}_{k_{m}} \right) \right| \leq \left| \varphi(\overline{t}_{l}) - \varphi \left( t^{(l)}_{k_{m}} \right) \right| + \left| \varphi \left( t^{(l)}_{k_{m}} \right) - v_{k_{m}} \left( t^{(l)}_{k_{m}} \right)  \right|
	\end{equation}
	it follows that $v_{k_{m}}\left( t^{(l)}_{k_{m}} \right)$ convergence strongly to $\varphi(\overline{t}_{l})$ and, consequently, $\varphi( \overline{t}_{l} ) = \overline{v}_{l}$ that leads to a contradiction. Therefore, $v_{k_{m}}\left( t^{(l)}_{k_{m}} \right) \to \overline{v}_{l}$ strongly as $m \to \infty$ for each $l=-1,-2,\ldots$. Now as in the case of \textbf{(COM1)} we can construct the limit trajectory $v^{*}$.
\end{proof}
\begin{corollary}
	\label{COR: CompatTrajectories}
	Suppose that either \textbf{(COM2)} or \textbf{(COM1)} with \textbf{(H1)}, \textbf{(H2)} and \textbf{(H3)} are satisfied; then every bounded in the future semi-trajectory is compact, i.~e. \textbf{(COM3)} holds.
\end{corollary}
\begin{proof}
	Indeed, let $u(t)$, where $t \geq 0$, be a bounded trajectory at $q$. Suppose $s_{k} \to +\infty$ as $k \to +\infty$ and let us show that the sequence $u(s_{k})$ has a limit point in $\mathbb{H}$. Consider the trajectories $v_{k}(t):=u(t+s_{k})$, which are defined for $t \geq T_{k} :=-s_{k}$. For $l=-1,-2,\ldots$ put $t^{(l)}_{k}:=l$. Then the sequence $v_{k}(\cdot)$ along with $t^{(l)}_{k}$ satisfy the conditions of Lemma \ref{LEM: GeneralCompactnessLemma}. In particular, this implies that for some subsequence $k_{m}$, $m=1,2,\ldots$, the sequence $u(s_{k_{m}})=v_{k_{m}}(0)$ converges.
\end{proof}

\begin{proof}[Proof of Theorem \ref{TH: SmithTheorem}]
	Let $u^{*}(\cdot)$ be an amenable trajectory at $q$ from \textbf{(H4)}. Let $\zeta \in \mathbb{H}^{-}$. We are going to construct an amenable trajectory $v^{*}(\cdot)$ at $q$ with the property $\Pi v^{*}(0) = \zeta$. For $k=1,2,\ldots$ with the use of the operator $G_{T_{1}}^{T_{2}}$ defined in \eqref{EQ: OperatorG} we consider the trajectories
	\begin{equation}
	v_{k}(t) := \psi^{t+k}\left(\vartheta^{-k}(q), \left( G_{-k}^{0} \right)^{-1}\left( \zeta ; q, u^{*}(-k) \right) \right)
	\end{equation}
	defined for $t \geq -k$. By the construction and from \textbf{(H3)} we have $V(v_{k}(t) - u^{*}(t)) \leq e^{-2\nu k)}V(v_{k}(-k)-u(-k)) \leq 0$ for all $t \geq -k$ and $\Pi v_{k}(0) = \zeta$. Analogously to \eqref{EQ: OperatorGKeyIneq} from \textbf{(H3)} we get
	\begin{equation}
	\label{EQ: SmithThIntegral}
	\delta^{-1} \cdot \left|\restr{P}{\mathbb{H}^{-}}\right| \cdot \left|\zeta - \Pi u^{*}(0) \right|^{2} \geq \int\limits_{-k}^{0}e^{2\nu s} \left|v_{k}(s)-u^{*}(s)\right|^{2}ds.
	\end{equation}	
	From considering the integral in \eqref{EQ: SmithThIntegral} on intervals $I_{l} := [l,l+1]$, $l=-1,-2,\ldots$, it follows that for every negative integer $l$ and for all sufficiently large $k$ there exists $t^{(l)}_{k} \in I_{l}$ such that the sequence $v_{k}(t^{(l)}_{k})$ is bounded in $\mathbb{H}$. Thus, we are in the conditions of Lemma \ref{LEM: GeneralCompactnessLemma}, which guarantees that there is a complete trajectory $v^{*}$ such that for some subsequence of $v_{k}$'s we have $v_{k_{m}}(t) \to v^{*}(t)$ as $m \to +\infty$ for all $t \in \mathbb{R}$. Since we have $V(u^{*}(t)-v_{k_{m}}(t)) \leq 0$ for all $t \geq -k_{m}$, it follows that $V(u^{*}(t)-v^{*}(t)) \leq 0$ for all $t \in (-\infty,+\infty)$. By Lemma \ref{LEM: AmenableLemma}, the complete trajectory $v^{*}$ is amenable and it is clear that $\Pi v^{*}(0)=\zeta$. Since $\zeta \in \mathbb{H}^{-}$ was arbitrary we get $\Pi \mathfrak{A}(q) = \mathbb{H}^{-}$ and, by Lemma \ref{LEM: PiIsHomeo}, $\Pi$ is a homeomorphism. Thus the proof is finished.
\end{proof}

\begin{remark}
	Despite that we were considering $j>0$ in \textbf{(H2)} some simple results hold in the case $j=0$, i.~e. then the subspace $\mathbb{H}^{-}$ is zero-dimensional. From Lemma \ref{LEM: AmenableLemma} it is easy to see that in the case $j=0$ there may be only one amenable trajectory. So, if it exists the statement of Theorem \ref{TH: SmithTheorem} still takes place and it holds without any assumptions of compactness.
\end{remark}
\section{The map $\Phi$ and inertial manifold $\mathfrak{A}$}
\label{SEC: TheMapPhi}

In this section we suppose that \textbf{(H1)}, \textbf{(H2)}, \textbf{(H3)} and \textbf{(H4)} are satisfied and $\Pi \mathfrak{A}(q) = \mathbb{H}^{-}$. Then it follows that the map $\Pi_{q} := \restr{\Pi}{\mathfrak{A}(q)} \colon \mathfrak{A}(q) \to \mathbb{H}^{-}$ is a homeomorphism for any $q \in \mathcal{Q}$. Consider the function $\Phi(q,\zeta) := \Pi^{-1}_{q} (\zeta) \in \mathfrak{A}(q)$. If $u(\cdot)$ is an amenable trajectory at $q$ it is clear that $u(t)=\Phi(\vartheta^{t}(q),\Pi u(t))$ for $t \in \mathbb{R}$. We state here an open problem linked with the continuity of $\Phi$.

\begin{problem}
When the map $\Phi \colon \mathcal{Q} \times \mathbb{H}^{-} \to \mathbb{H}$ defined above is continuous?
\end{problem}

In \cite{Anikushin2019ND} the author showed the continuity of $\Phi$ for cocycles generated by a certain class of nonlinear almost periodic ODEs (the idea can be used for certain infinite-dimensional systems). Here we give a positive solution to the problem for cocycles over the linear flow on $\mathbb{R}$, periodic cocycles and semi-flows.

\begin{proposition}
	\label{PROP: LinearFlowOnRCont}
	Suppose that the driving system $(\mathcal{Q}, \{ \vartheta^{t} \})$ is the shift on $\mathbb{R}$, i.~e. $\mathcal{Q}=\mathbb{R}$ and $\vartheta^{t}(q) = q + t$ for all $t \in \mathbb{R}$ and $q \in \mathcal{Q}$. Then the map $\Phi$ is continuous.
\end{proposition}
\begin{proof}
	In this case we may consider $\Phi$ as a function of $(t,\zeta)$. Suppose $\zeta_{k} \to \overline{\zeta}$ in $\mathbb{H}^{-}$ and $t_{k} \to \overline{t}$ in $\mathbb{R}$. Let $u_{k}(\cdot)$ and $u^{*}(\cdot)$ be amenable trajectories at $0$ such that $\Pi u_{k}(t_{k}) = \zeta_{k}$ and $\Pi u^{*}(\overline{t}) = \overline{\zeta}$. We have to show that $u_{k}(t_{k})=\Phi(t_{k},\zeta_{k}) \to u^{*}(\overline{t}) = \Phi(\overline{t},\overline{\zeta})$. Assuming the contrary, we get a subsequence and a number $\delta_{0} > 0$ such that $|u_{k_{m}}(t_{k_{m}}) - u^{*}(\overline{t})| \geq \delta_{0} > 0$. In the inequality
	\begin{equation}
	|u_{k_{m}}(t_{k_{m}}) - u^{*}(\overline{t})| \leq | u_{k_{m}}(t_{k_{m}}) - u^{*}(t_{k_{m}}) | + |u^{*}(t_{k_{m}}) - u^{*}(\overline{t})|
	\end{equation}
	the second term in the right-hand side tends to zero since $u^{*}(\cdot)$ is continuous. From \textbf{(H3)}, \textbf{(H1)}, the amenability and the Cauchy-Bunyakovsky-Schwarz inequality we get
	\begin{equation}
	\delta^{-1} \cdot \left| \restr{P}{\mathbb{H}^{-}} \right| \cdot |\zeta_{k_{m}} - \overline{\zeta} |^{2} \geq \int\limits_{-\infty}^{t_{k_{m}}} e^{2\nu s} |u_{k_{m}}(s)-u^{*}(s)|ds.
	\end{equation}
	Repeating the same argument as in Lemma \ref{LEM: PiIsHomeo} we get the convergence of $u_{k_{m}}(t_{k_{m}}) \to u^{*}(\overline{t})$ that leads to a contradiction.
\end{proof}

\begin{remark}
	It is clear that the arguing in Proposition \ref{PROP: LinearFlowOnRCont} is applicable if $(\mathcal{Q},\vartheta)$ is a minimal $\sigma$-periodic flow, i.~e. $\mathcal{Q} = \mathbb{R}/\sigma \mathbb{Z}$ and $\vartheta^{t}(q) = q + t$, or if $\mathcal{Q}$ is one point, i.~e. $\psi$ is a semi-flow in $\mathbb{H}$.
\end{remark}

The continuity of $\Phi$ is linked with some nice properties of the cocycle that we state below.

\begin{proposition}
	\label{EQ: Prop1FuncPhi}
	Suppose $\mathcal{Q}$ is compact and $\Phi$ is continuous. Then any bounded complete trajectory is compact.
\end{proposition}
\begin{proof}
	Suppose $u^{*}$ is a bounded complete trajectory of the cocycle at $q \in \mathcal{Q}$. Let $t_{k} \in \mathbb{R}$ be any sequence of real numbers. We consider a subsequence such that $\zeta_{k_{m}}:=\Pi u(t_{k_{m}})$ converges to some $\overline{\zeta} \in \mathbb{H}^{-}$ and $\vartheta^{t_{k_{m}}}(q)$ converges to some $\overline{q} \in \mathcal{Q}$ as $m \to \infty$. Since $\Phi$ is continuous we have $\Phi(\vartheta^{t_{k_{m}}}(q),\zeta_{k_{m}}) \to \Phi(\overline{q},\overline{\zeta}) \in \mathfrak{A}(\overline{q})$. But $u^{*}$ is bounded and therefore it is amenable, so $u(t_{k_{m}})=\Phi(\vartheta^{t_{k_{m}}}(q),\zeta_{k_{m}})$ converges.
\end{proof}

Consider the set $\mathfrak{A} := \bigcup\limits_{q \in \mathcal{Q}} \{q\} \times \mathfrak{A}(q) \subset \mathcal{Q} \times \mathbb{H}$ that we call \textit{complete amenable set}. Clearly, the set $\mathfrak{A}$ is invariant w. r. t. the skew-product flow $\pi^{t} \colon \mathcal{Q} \times \mathbb{H} \to \mathcal{Q} \times \mathbb{H}$ defined as $\pi^{t} (q, u) := (\vartheta^{t}(q),\psi^{t}(q,u))$ for $t \geq 0$. Thus, $\pi^{t} (\mathfrak{A}) = \mathfrak{A}$ and, in virtue of Corollary \ref{COR: UniqueAmenable}, $\pi^{t}$ is bijective on $\mathfrak{A}$. From the proof of Proposition \ref{EQ: Prop1FuncPhi} it is clear that the continuity of $\Phi$ implies the closedness of the set $\mathfrak{A}$.

\begin{proposition}
	Suppose $\mathcal{Q}$ is a topological manifold without boundary. Then the function $\Phi$ is continuous iff the set $\mathfrak{A}$ is a topological manifold without boundary.
\end{proposition}
\begin{proof}
	The map $h \colon \mathfrak{A} \to \mathcal{Q} \times \mathbb{H}^{-}$ defined as $(q,u) \mapsto (q,\Pi(u))$ is a continuous bijection. The continuity of $\Phi$ implies the continuity of $h^{-1}$ and therefore $\mathfrak{A}$ in its natural topology can be endowed with the manifold structure induced from $\mathcal{Q} \times \mathbb{H}^{-}$. If we know that $\mathfrak{A}$ is a topological manifold without boundary then, by the Brouwer theorem on invariance of domain applied to $h$, the inverse map $h^{-1} \colon (q,\zeta) \to (q,\Phi(q,\zeta))$ is continuous and so is $\Phi$.
\end{proof}

The following proposition shows that the complete amenable set $\mathfrak{A}$ may attract compact semi-trajectories that is related to properties of inertial manifolds. In \cite{Smith1994} R.~A.~Smith under somewhat different conditions showed an exponential attraction for $\mathfrak{A}$ in the case of autonomous reaction-diffusion systems (see Corollary 2 in \cite{Smith1994}).

\begin{proposition}
	\label{PROP: AttractionofA}
	Let $\mathcal{Q}$ be compact and $\Phi$ be continuous. Suppose that the semi-trajectory $u(t)=\psi^{t}(q,u_{0})$, where $t \geq 0$, is compact. Then we have
	\begin{equation}
	\operatorname{dist}\left( u(t),\mathfrak{A}\left( \vartheta^{t}(q) \right) \right) \to 0 \text{ as } t \to +\infty.
	\end{equation}
	Moreover, $|u(t)-\Phi( \vartheta^{t}(q), \Pi u(t) )| \to 0$ as $t \to +\infty$.
\end{proposition}
\begin{proof}
	It is enough to prove only the last statement. If we suppose the contrary then there is a number $\delta_{0}>0$ and a sequence $t_{k}$, where $k=1,2,\ldots$, tending to $+\infty$ such that for all $k$ we have
	\begin{equation}
	\label{EQ: TakeMeLimitInertial}
	\left|u(t_{k}) - \Phi(\vartheta^{t_{k}}(q),\Pi u(t_{k})) \right| \geq \delta_{0} > 0.
	\end{equation}
    Since $u(\cdot)$ and $\mathcal{Q}$ are compact we may assume that $u(t_{k}) \to \overline{v}_{0} \in \mathbb{H}$ and $\vartheta^{t_{k}}(q) \to \overline{q} \in \mathcal{Q}$. It is clear that $\overline{v}_{0} \in \mathfrak{A}(\overline{q})$ and therefore $\Phi(\overline{q},\Pi \overline{v}_{0}) = \overline{v}_{0}$. From this, taking it to the limit in \eqref{EQ: TakeMeLimitInertial} as $k \to +\infty$ and using the continuity of $\Phi$ we get a contradiction. 
\end{proof}

\section{Abstract periodic cocycles}
\label{SEC: AbstractPerCoc}
In this section we suppose that the flow $\vartheta$ on $\mathcal{Q}$ is minimal $\sigma$-periodic, i.~e. it is topologically conjugate to a linear flow $\vartheta_{\sigma}^{t}$ on $\mathcal{S}^{1}_{\sigma}=\mathbb{R}/\sigma\mathbb{Z}$ defined as $\vartheta_{\sigma}^{t}(\theta):= \theta + t$, $\theta \in \mathcal{S}^{1}_{\sigma}$. Our aim is to prove Theorems \ref{TH: ConvergenceTh}, \ref{TH: StablePeriodicSmith} and \ref{TH: MasseraTheorem}.

In what follows we are mainly deal with the case of \textbf{(H2)} with $j=1$. It is convenient to identify the one-dimensional subspace $\mathbb{H}^{-}$ with $\mathbb{R}$ to make the orthogonal projector $\Pi$ be a scalar-valued function. In order to prove Theorem \ref{TH: ConvergenceTh} we have to establish several lemmas.
\begin{lemma}
	\label{LEM: ConvergenceLemma}
	Suppose \textbf{(H1)}, \textbf{(H2)} with $j=1$ and \textbf{(H3)} hold. Then we have:
	\begin{enumerate}
		\item[1)] For any bounded in the future amenable trajectory $u^{*}$ (passing through $u^{*}(0)$ at $q$) there exists a $\sigma$-periodic trajectory $v^{*}$ (passing through $v^{*}(0)$ at $q$) such that $u^{*}(t)-v^{*}(t) \to 0$ as $t \to +\infty$.
		\item[2)] For any bounded in the past amenable trajectory $u^{*}$ (passing through $u^{*}(0)$ at $q$) there exists a $\sigma$-periodic trajectory $v^{*}$ (passing through $v^{*}(0)$ at $q$) such that $u^{*}(t)-v^{*}(t) \to 0$ as $t \to -\infty$.
		\item[3)] Let $u(t):=\psi^{t}(q,u_{0})$ be a bounded semi-trajectory passing through $u_{0}=u(0)$ at $q$, that remains in a compact subset $\mathcal{K}$ for $t \geq 0$. Then $u(t)-u(t+\sigma) \to 0$ as $t \to +\infty$.
	\end{enumerate} 
\end{lemma}
\begin{proof}
	1) If $u^{*}$ is not $\sigma$-periodic then the difference $\Pi u^{*}(t) - \Pi u^{*}(t+\sigma)$ cannot be zero and since $j=1$ it must have constant sign. From this it follows that the sequence $\Pi u^{*}(k\sigma)$, where $k=1,2,\ldots$, is bounded and monotone. In particular, it is fundamental and since $\Pi_{q}$ is a homeomorphism (by Lemma \ref{LEM: PiIsHomeo}) the sequence $u^{*}(k\sigma)$ is also fundamental in $\mathbb{H}$. Denote its limit by $v^{*}_{0}$ and consider $v^{*}(t)=\psi^{t}(q,v^{*}_{0})$ for $t \geq 0$. By the continuity of the cocycle we get $u^{*}(t+k\sigma)-v^{*}(t) \to 0$ as $k \to \infty$. It is easy to see that the definition
	\begin{equation}
	v^{*}(s):=\lim\limits_{k \to +\infty}u^{*}(s+k\sigma)
	\end{equation}
	is correct for $s \in \mathbb{R}$ and defines a $\sigma$-periodic trajectory of the cocycle such that $u^{*}(t)-v^{*}(t) \to 0$ as $t \to +\infty$.
	
	2) As in 1) we have that the sequence $\Pi u^{*}(k\sigma)$, where $k=-1,-2,\ldots$, is bounded monotone and fundamental in $\mathbb{H}$. The required $\sigma$-periodic trajectory $v^{*}$ passes at $q$ through its limit.
	
	3) Case 1: Suppose that $V(u(t_{0})-u(t_{0}+\sigma)) < 0$ for some $t_{0} \geq 0$. From \textbf{(H3)} it follows that $e^{2\nu t} V(u(t)-u(t+\sigma))$ is non-decreasing and, consequently, $V(u(t)-u(t+\sigma)) < 0$ for all $t \geq t_{0}$. From \textbf{(H3)} for $t \geq t_{0}$ we get that
	\begin{equation}
	\label{EQ: ConvTh1FirstIneq}
	e^{2\nu t}\left|\restr{P}{\mathbb{H}^{-}}\right| \cdot |\Pi u(t) - \Pi u(t+\sigma)|^{2} \geq \delta \int_{t_{1}}^{t}e^{2\nu s}|u(s)-u(s+\sigma)|^{2}ds.
	\end{equation}
	It is clear that the function $\Pi u(t) - \Pi u(t+\sigma)$ is of constant sign for $t \geq t_{0}$. Therefore the sequence $\Pi u(t_{0}+k\sigma)$, where $k=1,2,\ldots$, is monotone and bounded since $u$ lies in $\mathcal{K}$. Hence the series $\sum_{k=1}^{\infty} |\Pi u(t_{1}+k\sigma) - \Pi u(t_{1}+k\sigma+\sigma)|$ converges and from \eqref{EQ: ConvTh1FirstIneq} we have
	\begin{equation}
	\begin{split}
	+\infty&> \delta^{-1}\left|\restr{P}{\mathbb{H}^{-}}\right| \cdot \sum_{k=1}^{\infty}|\Pi u(t_{1}+k\sigma)-\Pi u(t_{1}+k\sigma + \sigma)|^{2} \geq\\
	&\geq \sum_{k=1}^{\infty} e^{-2\nu (t_{1}+k\sigma)}\int_{t_{1}}^{t_{1}+k\sigma}e^{2\nu s}|u(s)-u(s+\sigma)|^{2}ds \geq \\
	&\geq e^{-\sigma} \sum_{k=1}^{\infty} \int_{t_{1}+k\sigma - \sigma}^{t_{1}+k\sigma} |u(s)-u(s+\sigma)|^{2} = e^{-\sigma} \int_{t_{1}}^{\infty}|u(s)-u(s+\sigma)|^{2}ds.
	\end{split}	
	\end{equation}
	Therefore the integral $\int_{0}^{\infty}|u(s)-u(s+\sigma)|^{2}ds$ converges. Let us show that from this it follows that $u(t)-u(t+\sigma) \to 0$ as $t \to +\infty$. Indeed, there is a sequence $t_{k}=k\sigma+\theta_{k}$, where $\theta_{k} \in [0,\sigma)$ such that $u(t_{k})-u(t_{k}+\sigma) \to 0$ as $k \to +\infty$. It is clear that $u(s+t_{k})-u(s+t_{k}+\sigma) = \psi^{s}(\vartheta^{k\sigma + \theta_{k}}(q),u(t_{k}))-\psi^{s}(\vartheta^{k\sigma+\sigma +\theta_{k}}(q),u(t_{k}+\sigma))=\psi^{s}(\vartheta^{\theta_{k}}(q),u(t_{k}))-\psi^{s}(\vartheta^{\theta_{k}}(q),u(t_{k}+\sigma))$. Since the cocycle is continuous, it is uniformly continuous on $[0,\sigma] \times \mathcal{Q} \times \mathcal{K}$ and we get that $u(t)-u(t+\sigma) \to 0$ as $t \to +\infty$.
	
	3) Case 2: Suppose that $V(u(t)-u(t+\sigma)) \geq 0$ for all $t \geq 0$. Then from \textbf{(H3)} we get that for any $t \geq 0$
	\begin{equation}
	V(u(0)-u(\sigma)) \geq \delta \int_{0}^{t} e^{2\nu s}|u(s)-u(s+\sigma)|^{2}ds.
	\end{equation}
	It is clear that the integral $\int_{0}^{+\infty} e^{2\nu s}|u(s)-u(s+\sigma)|^{2}ds$ converges and we use the convergence of $\int_{0}^{+\infty} |u(s)-u(s+\sigma)|^{2}ds$ as in 2.1) to show the required statement.
\end{proof}

Now for any semi-trajectory $u(t)=\psi^{t}(q,u_{0})$ we consider the corresponding semi-trajectory of the skew product flow $\pi^{t}$ (see Section \ref{SEC: TheMapPhi}) given as $\gamma(t) = (\vartheta^{t}(q),u(t))$, $t \geq 0$. We denote by $\omega(\gamma_{0})$, where $\gamma_{0}=\gamma(0)$, the $\omega$-\textit{limit set} of $\gamma_{0}$ (or $\gamma$) in $\mathcal{Q} \times \mathbb{H}$. Let $\omega_{q}(\gamma_{0}) \subset \{ q\} \times \mathfrak{A}_{q}$ denote its fibre over $q \in \mathcal{Q}$, i.~e. $\omega(\gamma_{0}) = \bigcup_{q \in \mathcal{Q}} \omega_{q}(\gamma_{0})$.

\begin{lemma}
	\label{LEM: OmegaLimitSet}
	Suppose \textbf{(H1)}, \textbf{(H2)} with $j=1$ and \textbf{(H3)} hold. Then the $\omega$-limit set of any compact semi-trajectory $\gamma$ consists of $\sigma$-periodic trajectories. Moreover, the fibres $\omega_{q}(\gamma_{0})$ are homeomorphic to closed segments of $\mathbb{R}$.
\end{lemma}
\begin{proof}
	Recall that the $\omega$-limit set of a compact semi-trajectory of a semi-flow consists of compact complete trajectories (see, for example, \cite{CarvalhoLangaRobinson2012}). Since any semi-trajectory of the cocycle in virtue of Lemma \ref{LEM: ConvergenceLemma} satisfy $u(t)-u(t+\sigma) \to 0$ as $t \to +\infty$ it is obvious that for $\gamma(t)=(\vartheta^{t}(q),u(t))$, where $t \geq 0$, the set $\omega(\gamma_{0})$ consists of $\sigma$-periodic trajectories. From $\sigma$-periodicity and Corollary \ref{COR: UniqueAmenable} it follows that the fibres $\omega_{q}(\gamma_{0})$ are connected and they are compact since $\omega(\gamma_{0})$ is compact. From Lemma \ref{LEM: PiIsHomeo} we get that $\omega_{q}(\gamma_{0})$ is homeomorphic to $\Pi\omega_{q}(\gamma_{0})$ that is a compact and connected subset of $\mathbb{R}$, i.~e. it is a closed segment.
\end{proof}

Our next aim is to show that the fibres $\omega_{q}(\gamma_{0})$ consist of only one point. Remind that a trajectory $v^{*}$ at $q$ is called \textit{Lyapunov stable} if for every $\varepsilon>0$ there exists $\delta>0$ such that any trajectory $u(t)=\psi^{t}(q,u_{0})$ with $|u(0)-v^{*}(0)|<\delta$ satisfy $|u(t)-v^{*}(t)| < \varepsilon$ for all $t \geq 0$.

\begin{lemma}
	\label{LEM: LyapunovInstabilityLemma}
	Suppose \textbf{(H1)},\textbf{(H2)} with $j=1$, \textbf{(H3)} and one of \textbf{(COM1)} or \textbf{(COM2)} are satisfied. Let $v^{*}$ be a Lyapunov unstable $\sigma$-periodic trajectory at $q \in \mathcal{Q}$. Then for every sufficiently small $\varepsilon>0$ there exists an amenable trajectory $v^{\varepsilon}$ and a number $\theta_{\varepsilon} \in [0,\sigma]$ such that
	\begin{enumerate}
		\item[1)] $|v^{\varepsilon}(t)-v^{*}(t)| \leq \varepsilon$ for $t \in (-\infty,\theta_{\varepsilon})$,
		\item[2)] $|v^{\varepsilon}(\theta_{\varepsilon})-v^{*}(\theta_{\varepsilon})| = \varepsilon$,
		\item[3)] For any other $\sigma$-periodic trajectory $u^{*}$ at $q$ the value $\Pi u^{*}(\theta_{\varepsilon})$ does not lie between $\Pi v^{\varepsilon}(\theta_{\varepsilon})$ and $\Pi v^{*}(\theta_{\varepsilon})$.
	\end{enumerate}
\end{lemma}
\begin{proof}
	Let $\varepsilon>0$ be sufficiently small such that for every $k=1,2,\ldots$ there exists a trajectory $\widetilde{v}_{k}(t):=\psi^{t}(q,\widetilde{v}_{k}(0))$ and a number $T_{k} > 0$ such that
	\begin{enumerate}
		\item[a)] $|\widetilde{v}_{k}(0)-v^{*}(0)| < \frac{1}{k}$,
		\item[b)] $|\widetilde{v}_{k}(t)-v^{*}(t)| < \varepsilon$ for $t \in [0,T_{k})$
		\item[c)] $|\widetilde{v}_{k}(T_{k})-v^{*}(T_{k})|=\varepsilon$.
	\end{enumerate}
   Let $T_{k}=m_{k}\sigma + \theta_{k}$, where $m_{k} \in \mathbb{Z}$ and $\theta \in [0,\sigma)$, and consider $v_{k}(t):=\widetilde{v}_{k}(t+m_{k}\sigma)$ defined for $t \geq -m_{k}\sigma$. From item b) and $\sigma$-periodicity of $v^{*}$ we have $|v_{k}(t)-v^{*}(t)| \leq \varepsilon$ for all $t \in [-m_{k}\sigma,0]$. From Lemma \ref{LEM: GeneralCompactnessLemma} we can obtain a subsequence of $v_{k}$ (we keep the same index) that converges to some amenable trajectory $v^{\varepsilon}$ (the amenability follows at once from the boundedness of $v^{\varepsilon}$ on $(-\infty,0]$ ). We may assume also that $\theta_{k}$ converges to some $\theta_{\varepsilon} \in [0,\sigma]$. From the properties in items a), b), c) it is clear that for the chosen $v^{\varepsilon}$ and $\theta_{\varepsilon}$ we have items 1) and 2) of the lemma satisfied.
   
   To show the property in item 3) suppose that there exists a $\sigma$-periodic trajectory $u^{*}$ such that $\Pi u^{*}$ lies between $\Pi v^{\varepsilon}$ and $\Pi v^{*}$. From item a), $\sigma$-periodicity of $v^{*}$ and the continuity of the cocycle we have that $v_{k}(-m_{k}\sigma+\theta_{\varepsilon}) \to v^{*}(\theta_{\varepsilon})$ as $k \to \infty$ and, in particular, $\Pi v_{k}(-m_{k}\sigma+\theta_{\varepsilon}) \to \Pi v^{*}(\theta_{\varepsilon})$. Moreover, we also have $v_{k}(\theta_{\varepsilon}) \to v^{\varepsilon}(\theta_{\varepsilon})$. Therefore, the value $\Pi u^{*}(\theta_{\varepsilon})$ must lie between $\Pi v_{k}(-m_{k}\sigma+\theta_{\varepsilon})$ and $\Pi v_{k}(\theta_{\varepsilon})$ if $k$ is sufficiently large. From this it follows that the function $\Pi(u^{*}(s)-v_{k}(s))$ changes the sign as $s$ varies in $[-m_{k}\sigma,\theta_{\varepsilon}]$ and, consequently, there is $s_{k} \in (-m_{k}\sigma, \theta_{\varepsilon} ]$ such that $|\Pi u^{*}(s_{k})-\Pi v_{k}(s_{k})| = 0$. From \textbf{(H3)} (or \eqref{EQ: H3TrajectoryForm}) we get
   \begin{equation}
   \label{EQ: LemmaLyapunovUnstability}
   \begin{split}
   -e^{2\nu s_{k}} V(u^{*}(s_{k})-v_{k}(s_{k})) &+ e^{-2\nu m_{k}\sigma}V(u^{*}(0)-v_{k}(-m_{k}\sigma))\\ &\geq \delta \int_{-m_{k}\sigma}^{s_{k}} e^{2\nu s}|u^{*}(s)-v_{k}(s)|^{2}ds.
   \end{split}
   \end{equation}
   Since $v^{*}$ and $u^{*}$ are distinct amenable trajectories, by Lemma \ref{LEM: AmenableLemma} we have $V(u^{*}(0)-v^{*}(0)) < 0$ and, consequently, $V(u^{*}(0)-v_{k}(-m_{k}\sigma)) < 0$ for sufficiently large $k$. From this, \textbf{(H1)} and \eqref{EQ: LemmaLyapunovUnstability} it follows that
   \begin{equation}
   0=e^{2\nu s_{k}} \left| \restr{P}{\mathbb{H}^{-}} \right| \cdot \left| \Pi u^{*}(s_{k}) - \Pi v_{k}(s_{k}) \right|^{2} \geq \delta \int_{-m_{k}\sigma}^{s_{k}}e^{2\nu s}|u^{*}(s)-v_{k}(s)|^{2}ds.
   \end{equation}
   Thus, $u^{*}(s)-v_{k}(s)=0$ for $s \in [-m_{k}\sigma,s_{k}]$ and, consequently, for all $s \geq -m_{k}\sigma$ that leads to a contradiction.
\end{proof}

\begin{proof}[Proof of Theorem \ref{TH: ConvergenceTh}]
	From Lemma \ref{LEM: OmegaLimitSet} we obtain that the $\omega$-limit set of $\gamma_{0}=\gamma(0)$, where $\gamma(t)=(\vartheta^{t}(q),u(t))$, consists of $\sigma$-periodic trajectories and its fibres $\omega_{q}(\gamma_{0})$ homeomorphic to closed segments in $\mathbb{R}$. Suppose that the fibre $\omega_{q}(\gamma_{0})$ is a non-point segment. Then there exists a $\sigma$-periodic trajectory $u^{*}$ corresponding to its interior point. Since $u^{*}$ is a non-isolated (from both sides) $\sigma$-periodic trajectory it must be Lyapunov stable due to Lemma \ref{LEM: LyapunovInstabilityLemma}. But since $u^{*}$ is Lyapunov stable and lies in the $\omega$-limit set of $u$ it must be the only $\sigma$-periodic trajectory in the $\omega$-limit set. Indeed, there is a number $\delta>0$ such that any semi-trajectory $\widetilde{u}(t)=\psi^{t}(q,\widetilde{u}(0))$ with $|\widetilde{u}(0)-u^{*}(0)| < \delta$ satisfy $|\widetilde{u}(t)-u^{*}(t)| < \varepsilon$ for all $t \geq 0$. There is a sequence $t_{k} \to +\infty$ such that $u(t_{k}) \to u^{*}(0)$ and $\vartheta^{t_{k}}(q) \to q$ as $k \to +\infty$. Since the semi-trajectory $u$ is compact and the cocycle is continuous we may assume that $t_{k} = m_{k} \sigma$, where $m_{k} \in \mathbb{Z}_{+}$. Therefore, for all sufficiently large $k$ we must have $|u(m_{k} \sigma) - u^{*}(m_{k}\sigma)| < \delta$ and, consequently, $|u(t)-u^{*}(t)|<\varepsilon$ for all sufficiently large $t$. This proves that $\omega_{q}(\gamma_{0})$ is a one point set.
	
	Now let $u^{*}$ be the unique $\sigma$-periodic trajectory. It is clear that the sequence $u(k\sigma)$, $k=1,2,\ldots$, converges to $u^{*}(0)$ and, consequently, $u(k\sigma + s) \to u^{*}(s)$ for $s \in [0,\sigma]$. For any $t \geq 0$ let $t=k \sigma + s$, where $s \in [0,\sigma)$ and $k \in \mathbb{Z}_{+}$. Then we have $u(t)-u^{*}(t) = u(k \sigma + s)-u^{*}(s) \to 0$ as $t \to +\infty$. The theorem is proved.
\end{proof}

Let $v^{*}$ be an amenable trajectory passing through $v^{*}(0)$ at $q \in \mathcal{Q}$. We call $v^{*}$ \textit{amenable stable} if for every $\varepsilon>0$ there exists $\delta>0$ such that any amenable trajectory $u^{*}$ at $q$ with $|u^{*}(0)-v^{*}(0)| < \delta$ satisfy $|u^{*}(t)-v^{*}(t)| < \varepsilon$ for all $t \geq 0$. We have the following lemma.

\begin{lemma}
	\label{LEM: AmenStIsLyapunSt}
	Suppose \textbf{(H1)}, \textbf{(H2)} with $j=1$, \textbf{(H3)} and one of \textbf{(COM1)} or \textbf{(COM2)} are satisfied. Then any amenable stable $\sigma$-periodic trajectory $v^{*}$ is Lyapunov stable.
\end{lemma}
\begin{proof}
	Suppose that $v^{*}$ is not Lyapunov stable. By Lemma \ref{LEM: LyapunovInstabilityLemma} give us an amenable trajectory $v^{\varepsilon}$ with the properties 1)-3) for all sufficiently small $\varepsilon>0$. By Lemma \ref{LEM: ConvergenceLemma} there is a $\sigma$-periodic trajectory $w^{\varepsilon}$ such that $|v^{\varepsilon}(t)-w^{\varepsilon}(t)| \to 0$ as $t \to -\infty$. Since $v^{*}$ is amenable stable and $|v^{\varepsilon}(\theta_{\varepsilon})-v^{*}(\theta_{\varepsilon})|=\varepsilon$ the trajectory $v^{\varepsilon}$(t) does not converge to $v^{*}(t)$ as $t \to -\infty$. From this it follows that $v^{*}$ and $w^{\varepsilon}$ are distinct $\sigma$-periodic trajectories. Moreover, for $t \leq 0$ and $l=1,2,\ldots$ we have
	\begin{equation}
	\label{EQ: Lemma9Smith}
	|v^{*}(t)-w^{\varepsilon}(t)| \leq |v^{*}(t-l\sigma)-v^{\varepsilon}(t-l\sigma)| + |v^{\varepsilon}(t-l\sigma) - w^{\varepsilon}(t-l\sigma)|.
	\end{equation}
	Taking it to the limit in \eqref{EQ: Lemma9Smith} as $l \to +\infty$ and with the use of property 1) of $v^{\varepsilon}$ we get $|v^{*}(t)-w^{\varepsilon}(t)| \leq \varepsilon$.
	
	We have to show that $\Pi v^{\varepsilon}(t)$ and $\Pi w^{\varepsilon}(t)$ lies from the same side of $\Pi v^{*}(t)$ for all $t$. Indeed, if this is not true then for some $\delta_{0} > 0$ we must have $|\Pi v^{\varepsilon}(t)- \Pi w^{\varepsilon}(t)| \geq |\Pi v^{*}(t) - \Pi w^{\varepsilon}(t)| \geq \delta_{0} > 0$ for all $t \in \mathbb{R}$ that contradicts to the convergence $v^{\varepsilon}(t)-w^{\varepsilon}(t) \to 0$ as $t \to -\infty$. 
	
	Clearly, we have $w^{\varepsilon}(t) - v^{*}(t) \to 0$ as $\varepsilon \to 0+$. There exists a sequence $\varepsilon_{k} \to 0+$ as $k \to +\infty$ such that all $\Pi w^{\varepsilon_{k}}$ (and corresponding to them $\Pi y^{\varepsilon_{k}}$) lie on the same side from $\Pi v^{*}$. But this will contradict the property 3) in Lemma \ref{LEM: LyapunovInstabilityLemma} of $v^{\varepsilon_{1}}$ since $\Pi w^{\varepsilon_{k}}(\theta_{\varepsilon_{1}})$ for sufficiently large $k$ will lie between $\Pi v^{\varepsilon_{1}}(\theta_{\varepsilon_{1}})$ and $\Pi v^{*}(\theta_{\varepsilon_{1}})$. The lemma is proved.
\end{proof}

For investigation of stability properties it is convenient to introduce the following definitions. Let $u$ and $v$ be two distinct amenable trajectories at $q \in \mathcal{Q}$. Then $|\Pi u(t) - \Pi v(t)| > 0$ for every $t \in \mathbb{R}$. This implies that the real-valued function $\Pi u(t) - \Pi v(t)$ has a constant sign. In particular, for $v(t)=u(t+\sigma)$ this means that the sequence $\Pi u(k \sigma)$, where $k=1,2,\ldots$, is either decreasing or increasing provided that $u$ is not $\sigma$-periodic. We will call such $u$ \textit{decreasing} or \textit{increasing} respectively.

For $v_{0}$ we call the corresponding $\sigma$-periodic trajectory $v^{*}$ \textit{upper amenable stable} if either there is a sequence $v^{*}_{k}$, where $k=1,2,\ldots$, of $\sigma$-periodic trajectories such that the sequence $\Pi v^{*}_{k}(0)$ is strictly decreasing and $\Pi v^{*}_{k}(0) \to \Pi v^{*}(0)$ as $k \to \infty$ or there is $\delta>0$ such that for every $u_{0} \in \mathfrak{A}_{q}$ with $\Pi v^{*} < \Pi u_{0} < \Pi v^{*} + \delta$ we have that the corresponding amenable trajectory $u$ with $u(0)=u_{0}$ is decreasing. The notion of \textit{lower amenable stability} can be introduced analogously. It is clear that any upper and lower amenable stable $\sigma$-periodic trajectory is amenable stable.

\begin{proof}[Proof of Theorem \ref{TH: StablePeriodicSmith}]
	By definition of the sink $\mathcal{S}_{0}$ there exists an open set $\mathcal{G} \supset \mathcal{S}_{0}$ such that for every $u_{0} \in \mathcal{G}$ we have $\psi^{t}(q,u_{0}) \in \mathcal{S}_{0}$ for all $t \geq T(u_{0})$. Since $S_{0}$ is bounded, every semi-trajectory $\psi^{t}(q,u_{0})$ with $u_{0} \in \mathcal{G}$ converges to a $\sigma$-periodic trajectory, which is obviously lying in $\mathcal{S}_{0}$.
	
	Let $\mathcal{K}$ be the collection of all $v_{0} \in \mathfrak{A}_{q} \cap \mathcal{S}_{0}$ corresponding to $\sigma$-periodic solutions. By Lemma \ref{LEM: PiIsHomeo} the set $\mathcal{K}$ is compact. Let $v^{*}_{1}$ be a $\sigma$-periodic trajectory such that $\Pi v^{*}_{1}(0) = \sup \Pi \mathcal{K}$. Clearly, $v^{*}_{1}$ is upper amenable stable. From this it follows that the set $\mathcal{K}_{u} \subset \mathcal{K}$ consisting of all upper amenable stable $\sigma$-periodic trajectories is not empty. Moreover, the $\sigma$-periodic trajectory $v^{*}_{2}$ such that $\Pi v^{*}_{2}(0) = \inf \Pi \mathcal{K}_{u}$ is upper amenable stable. It is easy to see that $v^{*}_{2}$ must be lower amenable stable. Therefore, $v^{*}_{2}$ is amenable stable and by Lemma \ref{LEM: AmenStIsLyapunSt} it is Lyapunov stable.
\end{proof}

\begin{proposition}
	\label{PROS: IsolatedStable}
	Let \textbf{(H1)}, \textbf{(H2)} with $j=1$, \textbf{(H3)} and \textbf{(COM3)} hold. Then any isolated Lyapunov stable $\sigma$-periodic trajectory is asymptotically Lyapunov stable.
\end{proposition}
\begin{proof}
	Suppose we have an isolated $\sigma$-periodic Lyapunov stable trajectory $v^{*}$ at $q \in \mathcal{Q}$ that is not asymptotically stable. For every $\delta>0$ there is a trajectory $u_{\delta}$ at $q$ such that $|u_{\delta}(0)-v^{*}(0)|<\delta$ and $u_{\delta}(t) \not\to v^{*}(t)$ as $t \to +\infty$. Since $v^{*}$ is Lyapunov stable, the trajectories $u_{\delta}$ are bounded in the future for all sufficiently small $\delta>0$. By Theorem \ref{TH: ConvergenceTh} there is a $\sigma$-periodic trajectory such that $u_{\delta}(t) \to v^{*}_{\delta}(t)$ as $t \to +\infty$. By the choice of $u_{\delta}$ it is clear that $v^{*}_{\delta}$ and $v^{*}$ are distinct and $v^{*}_{\delta}(t) \to v^{*}(t)$ as $\delta \to 0+$. Therefore, $v^{*}$ is not isolated that leads to a contradiction.
\end{proof}

Now we prove an extension of the Massera second theorem \cite{Massera1950} as follows.

\begin{proof}[Proof of Theorem \ref{TH: MasseraTheorem}]
	Indeed, the existence of a compact semi-trajectory implies the existence of a compact complete trajectory\footnote{Consider the sequence $u_{k}(t) = u(t+\sigma k)$, where $k=1,2,\ldots$, of trajectories defined for $t \geq -\sigma k$. One can subtract a subsequence $u_{k_{m}}$ such that $u_{k_{m}}(-l)$ is convergent for any $l=1,2,\ldots$. It is clear that $u_{k_{m}}(\cdot)$ converges to a bounded complete trajectory of the cocycle.}, which is amenable since $\nu > 0$. Thus \textbf{(H4)} is satisfied and Theorem \ref{TH: SmithTheorem} is applicable. Since $\mathfrak{A}(q)=\mathfrak{A}(\vartheta^{\sigma}(q))$, the Poincar\'{e} map $T(v) := \psi^{\sigma}(q,v)$, $v \in \mathfrak{A}(q)$, is a self-map of $\mathfrak{A}(q)$. By Theorem \ref{TH: SmithTheorem}, the set $\mathfrak{A}(q)$ is homeomorphic to $\mathbb{R}^{2}$ and the existence of bounded trajectory implies that $T$ has a point with a convergent subsequence of its iterates. Therefore, by a topological argument as in \cite{Massera1950} (see \cite{Pliss1966} for a complete proof), $T$ has a fixed point corresponding to a $\sigma$-periodic trajectory.
\end{proof}
\section{Delayed feedback systems}
\label{SEC: DelayedSystems}
We start with the following scalar delay equation:
\begin{equation}
\label{EQ: DelayEq}
\begin{split}
&\dot{x}=-\lambda x(t) + b f(t,\upsilon(t)) + g(t),\\
&\upsilon(t) = \int_{-\tau}^{0} \rho(s) x(t+s)ds
\end{split}
\end{equation}
where $\lambda>0$ and $\tau >0$ are some constants; $b = \pm 1$; $g$ and $\rho$ are continuous functions and $f$ is a scalar continuously differentiable function\footnote{As an example one may take $f(t,\upsilon)=b_{1}(t) \frac{1}{1+e^{-\upsilon}} + b_{2}(t)$, where $b_{1}(t)$ is positive.} satisfying $0 \leq \frac{d}{d\upsilon}f(t,\upsilon) \leq \mu_{0}$ for all $t \in \mathbb{R}$, $\upsilon \in \mathbb{R}$. We also suppose that $g$ and $f$ are $\sigma$-periodic in $t$. We will study the problem \eqref{EQ: DelayEq} in the $L_{2}$ setting. Namely, let $\mathbb{H}:=\mathbb{R} \times L_{2}([-\tau,0];\mathbb{R})$. Put $\mathcal{D}(A):=\{ (x_{0},\phi) \in \mathbb{H} \ | \ \phi \in W^{1,2}(-\tau,0;\mathbb{R}) \text{ and } \phi(0)=x_{0} \}$ and consider the operator $A \colon \mathcal{D}(A) \subset \mathbb{H} \to \mathbb{H}$ defined as
\begin{equation}
(x,\phi) \mapsto \left(-\lambda x, \frac{\partial}{\partial s} \phi\right).
\end{equation}
Put $\Xi:=\mathbb{R}$ and consider the operator $B \colon \Xi \to \mathbb{H}$ defined as
\begin{equation}
\xi \mapsto (b \xi, 0).
\end{equation}
Define the operator $C \colon \mathbb{H} \to \mathbb{R}$ as
\begin{equation}
(x,\phi) \mapsto \int_{-\tau}^{0} \rho(s)\phi(s)ds.
\end{equation}
\begin{remark}
	\label{REM: DelayLimitations}
	It is more natural to consider $\rho = \delta_{-\tau}$, where $\delta_{-\tau}$ is the delta function at $-\tau$, that makes $C$ unbounded in $\mathbb{H}$. However, in this case we cannot apply the frequency theorem (Theorem \ref{TH: YakubovichLikhtarnikovC0}) since it is required that the quadratic form in \eqref{EQ: QuadraticFormRealDelay} must be continuous on $\mathbb{H} \times \Xi$.
\end{remark}
For $q \in \mathbb{R}$ we consider the abstract Cauchy problem in $\mathbb{H}$
\begin{equation}
\label{EQ: DelayAbstractPert}
\begin{split}
&\dot{u} = Au + Bf(t,Cu) + G(t),\\
&u(q)=u_{0} \in \mathbb{H},
\end{split}
\end{equation}
where $G(t)=(g(t),0) \in \mathbb{R} \times \mathbb{H}$. Problem \eqref{EQ: DelayAbstractPert} gives rise to the $\sigma$-periodic cocycle $(\psi, \vartheta)$ in $\mathbb{H}$, where $\vartheta^{t}=\vartheta^{t}_{\sigma}$ is acting on $\mathcal{Q} = \mathcal{S}^{1}_{\sigma} = \mathbb{R}/\sigma\mathbb{Z}$ and $\psi^{t}(q, u_{0}):=u(t+q,q,u_{0})$, where $u(s,q,u_{0})$, $s \geq q$, is a solution (in a generalized sense) to \eqref{EQ: DelayAbstractPert} with $u(q,q,u_{0})=u_{0}$ (see \cite{Webb1981} or \cite{Faheem1987}).

Below we give an analysis of all the conditions in Theorem \ref{TH: YakubovichLikhtarnikovC0}) that makes it possible to apply previous results (in particular, Theorem \ref{TH: ConvergenceTh}) to the cocycle $(\psi,\vartheta)$. In our special case the properties of the linear part can be calculated directly. For the general theory of linear delay equations we refer to \cite{BatkaiPiazzera2005}.

\subsection{Linear part}
\label{SUBSEC: LinearPart}
Let $\nu \geq 0$ and $\nu \not= \lambda$. Clearly, the operator $A+\nu I$ is the generator of a $C_{0}$-semigroup $e^{\nu t}G(t)$, $t \geq 0$, in $\mathbb{H}$, where $G(t)$ is given by
\begin{equation}
(x_{0},\phi_{0}) \overset{G(t)}{\mapsto} (x_{0} e^{-\lambda t}, \phi(t,\cdot)),
\end{equation}
where for $t \geq 0$ and $s \in [-\tau,0]$
\begin{equation*}
\phi(t,s):= \begin{cases}
x_{0}e^{-\lambda (t+s)}, \text{ if } t+s \geq 0,\\
\phi_{0}(t+s), \text{ if } -\tau \leq t+s < 0.
\end{cases}
\end{equation*}
Since the pair $(A+\nu I, B)$ is exponentially stabilizable (for example, by the feedback $\xi(t) = -\operatorname{sgn}b \cdot \nu x(t)$), it is $L_{2}$-controllable (see Appendix \ref{SEC: FrequencyTheorem}). For $\nu > \lambda$ we have the splitting $\mathbb{H} = \mathbb{H}^{s} \oplus \mathbb{H}^{u}$, where $\mathbb{H}^{s}:=\{ (x_{0},\phi_{0}) \in \mathbb{H} \ | \ x_{0} = 0 \}$ and $\mathbb{H}^{u} = \{ (x_{0},\phi_{0}) \in \mathbb{H} \ | \ \phi_{0}(s) = e^{(-\lambda + \nu)s} x_{0} \}$. Clearly, $\mathbb{H}=\mathbb{H}^{s}\oplus\mathbb{H}^{u}$ and for $u_{0}=(x_{0},\phi_{0}) \in \mathbb{H}^{s}$ we have $G(t)u_{0} \to 0$ as $t \to +\infty$ and for $u_{0} \in \mathbb{H}^{u}$ we have $G(t)u_{0} \to 0$ as $t \to -\infty$.

\subsection{Frequency-domain condition}
Consider the continuous quadratic form on $\mathbb{H} \times \Xi$
\begin{equation}
\label{EQ: QuadraticFormRealDelay}
F(u,\xi) = F(x,\phi,\xi):=\xi\left(\mu_{0} \int_{-\tau}^{0}\rho(s)\phi(s)ds - \xi\right)
\end{equation}
and its Hermitian extension (i. e. for $u \in \mathbb{H}^{\mathbb{C}}$, $\xi \in \Xi^{\mathbb{C}}$)
\begin{equation}
F^{\mathbb{C}}(u,\xi)=F^{\mathbb{C}}(x,\phi,\xi)=\operatorname{Re}\left[\xi^{*}\left(\mu_{0} \int_{-\tau}^{0}\rho(s)\phi(s)ds - \xi\right)\right].
\end{equation}
The choice of the form $F$ is related to properties of the nonlinearity $f$. Namely, the following important property is satisfied:
\begin{enumerate}
	\item[\textbf{(Q1)}] $F(u_{1}-u_{2}, \xi_{1}-\xi_{2}) \geq 0$ for all $u_{1},u_{2} \in \mathbb{H}$ and $\xi_{1} = f(t,Cu_{1})$, $\xi_{2}=f(t,Cu_{2})$.
\end{enumerate}

\begin{remark}
	\label{REM: ConstructingQuadraticForms}
	If we know some bound on the derivative of $f$ from below (in addition to the upper bound), i.~e. that $\mu_{1} \leq \frac{d}{d\nu}f(t,\sigma) \leq \mu_{2}$, then an appropriate choice of the quadratic form will be
		\begin{equation}
		F(u,\xi) = F(x,\phi,\xi):=\left(\xi - \mu_{1} \int_{-\tau}^{0}\rho(s)\phi(s)ds\right)\left(\mu_{2} \int_{-\tau}^{0}\rho(s)\phi(s)ds - \xi\right).
		\end{equation}
	However, one should require that for some $\mu_{0} \in [\mu_{1},\mu_{2}]$ the operator $(A + \nu I + \mu_{0}BC)$, where $C (x,\phi) := \int_{-\tau}^{0}\rho(s)\phi(s)ds$, admits a dichotomy, which allows to determine desired properties of $P$ from \eqref{EQ: FreqFormIneq} with $\xi = \mu_{0}C$. The nonlinearities with unbounded derivatives may also be considered \cite{LeoBurShep1996, AnikushinRR2019}. In this case one has to use the so called frequency theorem for the degenerate case \cite{Likhtarnikov1976,Likhtarnikov1977,LouisWexler1991}. For applications of this theorem in infinite dimensions there is a problem on checking the semiboundedness of a certain quadratic functional. However, in some cases \cite{ArovYakub1982} the mentonied semiboundedness is equivalent to the non-strict frequency-domain condition (as in the finite-dimensional case).
\end{remark}

By $A^{\mathbb{C}}$, $B^{\mathbb{C}}$ and $C^{\mathbb{C}}$ we denote the complexifications of the operators $A,B$ and $C$ respectively. Now consider the \textit{transfer function} of the triple $(A,B,C)$, i. e. $W(p):=C^{\mathbb{C}}(A^{\mathbb{C}}-pI)^{-1}B^{\mathbb{C}}$. Clearly, the spectrum of $A$ consist of a single eigenvalue $-\lambda$ and, consequently, the transfer function is defined for all $p \in \mathbb{C}$, $p \not = -\lambda$. For every such $p$ the function $W(p)$ is a linear operator $\mathbb{C} \to \mathbb{C}$ and therefore it can be identified with a complex number. It is easy to check that
\begin{equation}
\label{EQ: DelayTransferFunction}
W(p)=-\frac{b}{\lambda + p}\int_{-\tau}^{0}\rho(s) e^{p s} ds.
\end{equation}
Now let $\nu \geq 0$, $\nu \not= \lambda$. We are going to state the frequency domain condition (the condition $\alpha_{3}<0$ from Theorem \ref{TH: YakubovichLikhtarnikovC0}) for the control system given by the pair $(A+\nu I, B)$ and the form $F(u,\xi)$. Since
\begin{equation}
\label{EQ: DelayFreqDomCond}
\begin{split}
F^{\mathbb{C}}(-(A+\nu I - i\omega I)^{-1}B\xi,\xi) &= \operatorname{Re} \left[ \xi^{*} ( -\mu_{0}W(i\omega - \nu)\xi - \xi ) \right] \\
&= -\mu_{0} |\xi|^{2} \left(\operatorname{Re} W(i\omega - \nu) + \frac{1}{\mu_{0}} \right).
\end{split}
\end{equation}
the frequency-domain condition is
\begin{enumerate}
	\item[\textbf{(DF)}] $\operatorname{Re} W(i\omega - \nu) + 1/\mu_{0} > 0$ for all $\omega \in \mathbb{R}$.
\end{enumerate}

\begin{theorem}
	\label{TH: DelayedFeedback}
	Suppose \textbf{(DF)} with some $\nu > \lambda$ is satisfied. Then there exists a self-adjoint operator $P \in \mathcal{L}(\mathbb{H})$ such that \textbf{(H1)}, \textbf{(H2)} with $j=1$ and \textbf{(H3)} are satisfied for the cocycle $(\psi,\vartheta)$ generated by \eqref{EQ: DelayAbstractPert}. 
\end{theorem}
\begin{proof}
	We have that $A+\nu I$ is the generator of a $C_{0}$-semigroup and the pair $(A+\nu I,B)$ is $L_{2}$-controllable (see Subsection \ref{SUBSEC: LinearPart}). The frequency domain condition \textbf{(DF)} allows us to apply Theorem \ref{TH: YakubovichLikhtarnikovC0} to get a self-adjoint operator $P \in \mathcal{L}(\mathbb{H})$ such that
	\begin{equation}
	\label{EQ: FreqFormIneq}
	2 ((A+\nu I)u + B\xi, Pu) + F(u,\xi)  \leq -\delta (|u|^2 + |\xi|^{2}), \text{ for } u \in \mathcal{D}(A), \xi \in \mathbb{R}.
	\end{equation}
	Putting $\xi = 0$ in \eqref{EQ: FreqFormIneq} we get
	\begin{equation}
	2((A+\nu I)u,Pu) \leq -\delta |u|^{2}.
	\end{equation}
	Since zero is a regular value of $A+\nu I$, we have $\mathcal{R}(A+\nu I) = \mathbb{H}$ and by Proposition \ref{PROP: TrivialKernel} we get $\operatorname{Ker}P = \{ 0 \}$. Applying Propositions \ref{PROP: EstimatesSubspaces} and \ref{PROP: EstimateEigenvaluesWithStable} to $A+\nu I$, $P$ and the spaces $\mathbb{H}^{s}$ and $\mathbb{H}^{u}$ defined in Subsection \ref{SUBSEC: LinearPart} we get that $\dim\mathbb{H}^{-} = 1$.
	
	Putting $u=u_{1}-u_{2}$ and $\xi = f(Cu_{1})-f(Cu_{2})$ in \eqref{EQ: FreqFormIneq} we get
	\begin{equation}
	\begin{split}
	2(A(u_{1}-u_{2}) + B(f(Cu_{1})-f(Cu_{2})),P(u_{1}-u_{2}))+\\+F(u_{1}-u_{2},f(Cu_{1})-f(Cu_{2})) \leq -\delta |u_{1}-u_{2}|^{2}.
	\end{split}
	\end{equation}
	From this with the use of property \textbf{(Q1)} of the form $F$ we get
	\begin{equation}
	\label{EQ: FormProperty}
	2(A(u_{1}-u_{2}) + B(f(Cu_{1})-f(Cu_{2})),P(u_{1}-u_{2})) \\ \leq -\delta |u_{1}-u_{2}|^{2}.
	\end{equation}
	From Theorem 2.5 and Remark 2.2 in \cite{Webb1981} it follows that for initial values from $\mathcal{D}(A)$ we have strong solutions\footnote{According to \cite{Webb1981}, by a strong solution to \eqref{EQ: DelayAbstractPert} on $[l,r]$ we mean an $\mathbb{H}$-valued function $u(\cdot)$, which is continuous on $[l,r]$, absolutely continuous on compact subintervals of $(l,r)$ and satisfy \eqref{EQ: DelayAbstractPert} almost everywhere on $(l,r)$.} to \eqref{EQ: DelayAbstractPert}. Let $l \leq r$ be given and let $u_{1}(t)$ and $u_{2}(t)$ be two strong solutions to \eqref{EQ: DelayAbstractPert} for $t \in [l,r]$. From \eqref{EQ: FormProperty} we have for almost all $s \in [l,r]$
	\begin{equation}
	\label{EQ: BeforeIntegrating}
	\frac{d}{ds}[ e^{2\nu s}V(u_{1}(s)-u_{2}(s)) ] \leq -\delta e^{2\nu s}|u_{1}(s)-u_{2}(s)|^{2}.
	\end{equation}
	Integrating \eqref{EQ: BeforeIntegrating} on $[l,r]$ we get
	\begin{equation}
	\label{EQ: FinalIneqH3}
	e^{2 \nu r}V(u_{1}(r)-u_{2}(r)) - e^{2\nu l}V(u_{1}(l)-u_{2}(l)) \leq -\delta\int_{l}^{r}e^{2\nu s}|u_{1}(s)-u_{2}(s)|^{2}ds.
	\end{equation}
	Now we extend \eqref{EQ: FinalIneqH3} for solutions with initial data from $\mathbb{H}$ by continuity. Thus, we have \textbf{(H3)} satisfied.
\end{proof}

In fact, we have condition \textbf{(COM2)} with $t_{com}=2\tau$ satisfied (however, as it shown in \cite{Anikushin2019+OnCom}, the property \textbf{(COM1)} cannot hold in this case). This follows from Lemma 5.6 of \cite{Faheem1987}, linking the compactness of the evolution operator in $\mathbb{H}$ given by \eqref{EQ: DelayAbstractPert} with the compactness of the evolution operator in $C([-\tau,0];\mathbb{R})$ given by \eqref{EQ: DelayEq}. The latter property can be checked with the use of results of Section 3.6 in \cite{Hale} and, clearly, it holds in our case due to the Lipschitz property of the right-hand side in \eqref{EQ: DelayEq}. 

Summarizing the above, we obtain the following theorem.
\begin{theorem}
	\label{TH: ConvegenceDelayEqs}
	Let $(\psi,\vartheta)$ be the $\sigma$-periodic cocycle generated by \eqref{EQ: DelayAbstractPert} and suppose \textbf{(DF)} is satisfied. Then any bounded in the future trajectory converges to a $\sigma$-periodic trajectory.
\end{theorem}

\noindent Note that in terms of solutions the convergence given by Theorem \ref{TH: ConvegenceDelayEqs} holds in the uniform norm.

\begin{remark}
	\label{REM: LimitFreqCond}
	Let us formally consider the frequency-domain condition \textbf{(DF)} for the case $\rho=\delta_{-\tau}$, i.~e.
	\begin{equation}
	\label{EQ: LimitFreqCond}
	\operatorname{Re}W_{0}(-\nu + i \omega) + \frac{1}{\mu_{0}} > 0, \text{ for } \omega \in \mathbb{R},
	\end{equation}
	where $W_{0}(p)=-\frac{e^{- \tau p}}{\lambda + p}$. Even if \eqref{EQ: LimitFreqCond} holds we cannot apply Theorem \ref{TH: YakubovichLikhtarnikovC0} as we did in the proof of Theorem \ref{TH: DelayedFeedback} since the form $F$ in \eqref{EQ: QuadraticFormRealDelay} with $\rho = \delta_{-\tau}$ is not continuous on $\mathbb{H} \times \Xi$.
	Now we suppose that the function $\rho$ approximates the delta function $\delta_{-\tau}$ at $-\tau$. Namely, we take $\rho = \rho_{n} = n \cdot  \chi_{[-\tau,-\tau+1/n]}$, where $\chi_{\mathcal{A}}$ is the characteristic function of $\mathcal{A}$. Let $W_{n}(p)$ be the transfer function from \eqref{EQ: DelayTransferFunction} for $\rho = \rho_{n}$. It is clear that if \eqref{EQ: LimitFreqCond} is satisfied then for all sufficiently large $n$ we have \textbf{(DF)} with $W = W_{n}$ satisfied too.
	\begin{problem}
		How can we deal with the case $\rho=\delta_{-\tau}$ using this theory or its modifications (maybe with the consideration of unbounded operators)?
	\end{problem}
	
	Note that \eqref{EQ: LimitFreqCond} will be satisfied if $|W(i\omega - \nu)| < 1/\mu_{0}$, i.~e. $\frac{e^{\tau \nu}}{|\nu - \lambda|} < 1/\mu_{0}$. Putting $\nu=0$ or $\nu = \lambda + \tau^{-1}$, we arrive at the following conditions: 
	\begin{enumerate}
		\item[\textbf{(DF1)}] $\lambda > \mu_{0}$,
		\item[\textbf{(DF2)}] $-\tau e^{\tau \lambda + 1} + 1/\mu_{0} > 0$.
	\end{enumerate}
	Suppose the problem \eqref{EQ: DelayAbstractPert} is stationary, i.~e. $g \equiv 0$, $f$ is independent of $t$ and $f(0)=0$. If \textbf{(DF1)} is satisfied there is a self-adjoint positive definite operator $P \in \mathcal{L}(\mathbb{H})$ such that \textbf{(H3)} holds for a small $\nu > 0$. This condition implies that the stationary point $u \equiv 0$ is unique and may indicate its global asymptotic stability. If only \textbf{(DF2)} is satisfied there may be several stationary points and the conditions of global stability cannot be fulfilled. Thus, within the conditions of Theorem \ref{TH: ConvegenceDelayEqs} the appearance of several periodic trajectories is possible. 
\end{remark}

To study the existence of Lyapunov stable $\sigma$-periodic trajectories we have to construct a sink $\mathcal{S}_{0}$. For this one can use Theorem 5 from \cite{Smith1990} as follows.

\begin{theorem}
	Let the hypotheses of Theorem \ref{TH: ConvegenceDelayEqs} hold. Assume, in addition, that $f(t,v)$ is bounded uniformly in $t \in \mathbb{R}$ and $v \in \mathbb{R}$. Then the cocycle $(\psi,\vartheta)$ is dissipative\footnote{That is every trajectory enters a certain bounded set.}. Moreover, there exists at least one $\sigma$-periodic trajectory, which is Lyapunov stable.
\end{theorem}
\begin{proof}
	Consider the bounded linear operator $L \colon C([-\tau,0];\mathbb{R}) \to \mathbb{R}$ defined as $L\phi:=-\lambda \phi(0)$ and let $F(t,\phi):=-\lambda \phi(0) + b f(t,\int_{-\tau}^{0}\rho(s)\phi(s)ds)+g(t)$ be the right-hand side of \eqref{EQ: DelayEq}. It is clear that
	\begin{equation}
	\frac{| F(t,\phi) - L\phi |}{\| \phi \|_{\infty}} \to 0 \text{ as } \| \phi \|_{\infty} \to +\infty.
	\end{equation}
	From Theorem 5 \cite{Smith1990} it follows that the classic solutions of \eqref{EQ: DelayEq} are uniformly bounded, i.~e. there exist $R>0$ such that $|x(t,t_{0},x_{0})| \leq R$ for all $t \geq T(t_{0},x_{0})$, where $x(t,t_{0},x_{0})$ is a solution (with $x(t_{0},t_{0},x_{0})=x_{0}$) to \eqref{EQ: DelayEq} in the classic sense. Let us identify elements of $C([-\tau,0];\mathbb{R})$ with their images in $\mathbb{H}$ under the embedding $\phi \mapsto (\phi(0),\phi)$. Since cocycle $(\psi,\vartheta)$ given by the generalized solutions of \eqref{EQ: DelayAbstractPert} agrees on $C([-\tau,0];\mathbb{R})$ with the cocycle given by the classical solutions and $\psi^{\tau}(q,\mathbb{H}) \subset C([-\tau,0];\mathbb{R})$ (see \cite{Webb1981}), it is clear that $\mathcal{S}_{0}:=\{ u \in \mathbb{H} \ | \ |u| \leq R \sqrt{\tau+1} \}$ is a sink for $(\psi,\vartheta)$ with $\mathcal{G}:=\mathbb{H}$. Now the existence of a Lyapunov stable $\sigma$-periodic trajectory follows from Theorem \ref{TH: StablePeriodicSmith}.
\end{proof}

Note that if we restrict ourselves with the cocycle $(\widetilde{\psi},\widetilde{v})$ generated by the classical solutions of \eqref{EQ: DelayEq} in $C([-\tau,0];\mathbb{R})$, the Lyapunov stability of any $\sigma$-periodic trajectory w.~r.~t. the norm of $\mathbb{H}$ (i.~e. its stability as a trajectory of $(\psi,\vartheta)$) will imply its Lyapunov stability w.~r.~t. the uniform norm (i.~e. its stability as a trajectory of $(\widetilde{\psi},\widetilde{v})$). For details we refer to \cite{Anik2020PB}.


Similar techniques can be applied to study a system
\begin{equation}
\begin{split}
&\dot{x}_{k}(t) = -\lambda_{k} x_{k}(t) + \sum_{j=1}^{n} b_{k j} f_{j}(t,\upsilon_{j}(t)), \\
&\upsilon_{j}(t)=\int_{-\tau_{j}}^{0}\rho_{j}(s)x(t+s)ds, \ k=1,\ldots,n,
\end{split}
\end{equation}
with several monotone nonlinearities $f_{j}$ with $0 \leq \frac{d}{d\upsilon}f_{j} \leq \mu_{j}$ and delays $\tau_{j} \geq 0$. In \cite{Marcus1989} such systems were considered as a model of analog neural networks. Motivated by the above reduction results, one can propose that for $n \geq 3$ (or $n \geq 2$ for the periodic case) such systems may exhibit chaos. To obtain a frequency-domain condition the quadratic form $F(x,\phi,\xi)$, where $x=(x_{1},\ldots,x_{n})$, $\phi=(\phi_{1},\ldots,\phi_{n})$ and $\xi=(\xi_{1},\ldots,\xi_{n})$, may be chosen as a sum
\begin{equation}
F(x,\phi,\xi) = \sum_{j=1}^{n}\kappa_{j}\xi_{j}\left(\mu_{j}\int_{-\tau_{j}}^{0}\rho_{j}(s)\phi_{j}(s)ds - \xi_{j}\right)
\end{equation}
with some coefficients $\kappa_{j} \geq 0$.

\section{Parabolic problems with monotone nonlinear boundary controls}
\label{SEC: Parabolic}

In this section we consider the following parabolic problem
\begin{equation}
\label{EQ: ParabolicProblem}
\begin{split}
u_{t}(t,x)&=\alpha u_{xx}(t,x) - \beta u(t,x) + g_{0}(t,x), \ x \in (0,1), t>t_{0} \\
u(t_{0},x)&=u_{0}(x), x \in (0,1),\\
u_{x}(t,0)&=0, \ u_{x}(t,1)=-f(t,\upsilon(t)), t > t_{0},\\
\upsilon(t)&=\int_{0}^{1}\rho(x)u(t,x)dx, t \geq t_{0}.
\end{split}
\end{equation}
Here $\alpha>0$, $\beta>0$ are real numbers; $u_{0} \in L_{2}(0,1)$; $f$ is a continuously differentiable in $\upsilon$ function such that $0 \leq \frac{d}{d\upsilon} f(t,\upsilon) \leq \mu_{0}$ for all  $t \in \mathbb{R}$, $\upsilon \in \mathbb{R}$; $g_{0}$ is a continuous function. We also suppose that $g$ and $f$ are $\sigma$-periodic in $t$. We suppose that $\rho$ is a continuous function.

In order to write \eqref{EQ: ParabolicProblem} in the abstract form we put $\mathbb{H}=\mathbb{H}_{0}:=L_{2}(0,1)$, $\mathbb{H}_{1}:=W^{1,2}(0,1)$ and $\mathbb{H}_{-1} = \mathbb{H}_{1}^{*} = W^{-1,2}(0,1)$. It can be checked that the natural embedding $\mathbb{H}_{1} \subset \mathbb{H}_{0}$ is compact. We identify $\mathbb{H}$ with its dual and consider the dual embedding $\mathbb{H}_{0} \subset \mathbb{H}_{-1}$. By $(u,v)$ we denote the dual pairing between $u \in \mathbb{H}_{-1}$ and $v \in \mathbb{H}_{1}$, which coincides with the scalar product of $u$ and $v$ in $\mathbb{H}$ if $u \in \mathbb{H}$. The operator $A \in \mathcal{L}(\mathbb{H}_{1},\mathbb{H}_{-1})$ is defined by the bilinear form
\begin{equation}
\label{EQ: FormofA}
a(u,v)=(Au,v):=-\int_{0}^{1}\left(\alpha u_{x}(x)v_{x}(x) + \beta u(x) v(x)\right) dx
\end{equation}
for $u,v \in \mathbb{H}_{1}$. We also put $\Xi := \mathbb{R}$ and consider the operator $B \in \mathcal{L}(\Xi,\mathbb{H}_{-1})$ defined by
\begin{equation}
b(\xi,v)=(B\xi,v):=-\alpha \cdot \xi v(1), \text{ for } \xi \in \Xi, v \in \mathbb{H}_{1}.
\end{equation}

The function $g_{0}(x,t)$ gives rise to a functional on $\mathbb{H}_{1}$ as
\begin{equation}
g(t)[v]:=\int_{0}^{1}g_{0}(x,t)v(x)dx, \text{ for } v \in \mathbb{H}_{1}.
\end{equation}
Now consider the Cauchy problem for the abstract evolution equation
\begin{equation}
\label{EQ: AbstractParabolicProb}
\begin{split}
&\dot{u} = Au + Bf(t,Cu) + g(t),\\
&u(t_{0})=u_{0}.
\end{split}
\end{equation}

\noindent Let $W(t_{0},T)$ be the space of functions $u \colon (t_{0},T) \to \mathbb{H}$ such that $u \in L_{2}(t_{0},T;\mathbb{H}_{1})$ and $\dot{u} \in L_{2}(t_{0},T;\mathbb{H}_{-1})$ (see Chapter III of \cite{Lions1971} for a more precise treatment). There is a continuous embedding $W(t_{0},T) \subset C(t_{0},T;\mathbb{H})$ \cite{Lions1971}. A function $u \in W(t_{0},T)$, $u(t)=u(t,t_{0},u_{0})$, is called a \textit{weak solution} to \eqref{EQ: AbstractParabolicProb} on $[t_{0},T]$ if $u(t_{0})=u_{0}$ and the differential equation in \eqref{EQ: AbstractParabolicProb} is satisfied in the sense of $L_{2}(t_{0},T;\mathbb{H}_{-1})$.

\begin{proposition}
	\label{PROP: ParabCocycle}
	Under the above given assumptions, for any $u_{0} \in L_{2}(0,1)$ and $t_{0} \in \mathbb{R}$ there exists a unique weak solution $u(t,t_{0},u_{0})$, where $t \geq t_{0}$, to \eqref{EQ: AbstractParabolicProb}, which depend continuously on $(t_{0},u_{0})$ in the norm of $W(t_{0},T)$ for any $T>t_{0}$. In particular, the equality $\psi^{t}(t_{0},u_{0}):=u(t+t_{0},t_{0},u_{0})$ defines a $\sigma$-periodic cocycle in $\mathbb{H}$.
\end{proposition}
\begin{proof}
	Put $\upsilon(t):=\int_{0}^{1}\rho(x)u(t,x)dx$. For $s>t_{0}$ we consider the adjoint problem in $(t,x) \in (t_{0},s) \times (0,1)$:
	\begin{equation}
	\begin{split}
	-w_{t}(t,x)&=\alpha w_{xx}(t,x) - \beta w(t,x),\\
	w_{x}(t,0)&=w_{x}(t,1)=0,\\
	w(s,x)&=\rho(x).
	\end{split}
	\end{equation}
	Let $w(t,x;s)$ be a solution of the above problem. Multiplying the first equation in \eqref{EQ: ParabolicProblem} by $w(t,x;s)$ and integrating it on $(t_{0},s) \times (0,1)$, we get
	\begin{equation}
	\label{EQ: NonlinearIntegralEq}
	\begin{split}
	\upsilon(s)=\int_{0}^{1}u_{0}(x)w(t_{0},x;s)dx &- \alpha \int_{t_{0}}^{s}f(t,\upsilon(t))w(t,1;s)dt +\\&+\int_{t_{0}}^{s}\int_{0}^{1}g_{0}(t,x)w(t,x;s)dx dt.
	\end{split}
	\end{equation}
	By the Lipschitz property of $f$, for every $T>0$, $t_{0} < T$ and $u_{0} \in L_{2}(0,1)$ there exists a unique solution $\upsilon(t)=\upsilon(t,t_{0},u_{0})$, $t \in [t_{0},T]$, to \eqref{EQ: NonlinearIntegralEq}, which depend continuously on $(t_{0},u_{0})$ in the norm of $C(t_{0},T;\mathbb{R})$. For every such function $\upsilon(t)$ the linear problem
	\begin{equation}
	\begin{split}
	&\dot{u} = Au + Bf(t,\upsilon(t)) + g(t),\\
	&u(t_{0})=u_{0}.
	\end{split}
	\end{equation}
	has a unique weak solution $u(t)=u(t,t_{0},u_{0})$, where $t \in [t_{0},T]$, which depend continuously on $(t_{0},u_{0})$ in the norm of $W(t_{0},T)$ (see Chapter III in \cite{Lions1971}) and, in virtue of the embedding theorem, continuous dependence also holds in the norm of $C(t_{0},T;\mathbb{H})$. It can be verified that $u$ is a weak solution to \eqref{EQ: AbstractParabolicProb}. Thus, equation \eqref{EQ: AbstractParabolicProb} generates a $\sigma$-periodic cocycle in $\mathbb{H}$.
\end{proof}

\subsection{Linear part}
\label{SUBSEC: LinearPartParab}
By \eqref{EQ: FormofA} for any $\nu \in \mathbb{R}$ and $\lambda=\beta-\nu-\alpha$ we have
\begin{equation}
\label{EQ: Coercivity}
((A+\nu I)u,u) + \lambda (u,u) = -\alpha (u,u)_{1}.
\end{equation}
Therefore, the corresponding linear problem is well-posed and the corresponding to $A+\nu I$ form is regular in the sense of \cite{Likhtarnikov1976} (see, for example, Theorem 1.2, Chapter III in \cite{Lions1971}).

It is well-known that the spectrum of $A$ consists of the eigenvalues $\lambda_{k}=-\alpha \pi^{2} k^{2} - \beta$, where $k=0,1,\ldots$. In what follows we will apply Theorem 4 of \cite{Likhtarnikov1976} to the pair $(A+\nu I,B)$ with $\pi^{2}\alpha + \beta > \nu > \beta$. By Theorem 3.1 in \cite{Liu2003} the pair $(A+\nu I, B)$ is exponentially stabilizable (in the sense of \cite{Likhtarnikov1976}). 

Now we consider $A$ as an operator $\mathcal{D}(A) \subset \mathbb{H} \to \mathbb{H}$ with $\mathcal{D}(A) = W^{2,2}(0,1)$. Let $G(t)$, $t \geq 0$, be the $C_{0}$-semigroup in $\mathbb{H}$ generated by $A$. Then the operator $A+\nu I$ generates the semigroup $e^{\nu t} G(t)$. Let $\pi^{2}\alpha + \beta > \nu > \beta$. From the eigenvalue decomposition of $A$ we have for $A+\nu I$ a one-dimensional unstable subspace $\mathbb{H}^{u}:=\{ u_{0} \in \mathbb{H} \ | \ e^{\nu t}G(t)u_{0} \to 0 \text{ as } t \to -\infty \} = \{ u_{0} \in \mathbb{H} \ | \ |e^{\nu t}G(t)u_{0}|_{\mathbb{H}} \to \infty \text{ as } t \to +\infty  \}$ and its orthogonal complement is the stable subspace $\mathbb{H}^{s} = \{ u_{0} \in \mathbb{H} \ | \ G(t)u_{0} \to 0 \text{ as } t \to +\infty \}$.

\subsection{Frequency-domain condition}
We consider the quadratic form $F(u,\xi)$ for $u \in \mathbb{H}$, $\xi \in \Xi$ defined as (see also Remark \ref{REM: ConstructingQuadraticForms})
\begin{equation}
F(u,\xi):=\xi \left( \mu_{0}\int_{0}^{1}\rho(x)u(x)dx-\xi \right).
\end{equation}
and its Hermitian extension for $u \in \mathbb{H}^{\mathbb{C}}$, $\xi \in \Xi^{\mathbb{C}}$ is given by
\begin{equation}
F^{\mathbb{C}}(u,\xi) := \operatorname{Re}\left[\xi^{*} \left( \mu_{0}\int_{0}^{1}\rho(x)u(x)dx-\xi \right)\right].
\end{equation}
Clearly, the form $F$ satisfy an analog of \textbf{(Q1)}.

The transfer function $W(p):=C^{\mathbb{C}}(A^{\mathbb{C}}-pI^{\mathbb{C}})^{-1} B^{\mathbb{C}}$, where $p \in \mathbb{C}, p \not= \lambda_{k}$, is given by
\begin{equation}
W(p)=-\int_{0}^{1} \rho(x) \widetilde{u}(x,p) dx,
\end{equation}
where $\widetilde{u}$ is the solution of
\begin{equation}
\begin{split}
&\alpha u_{xx} - \beta u = p u,\\
&u_{x}(0,p)=0, u_{x}(1,p)=1.
\end{split}
\end{equation}
Clearly, we have
\begin{equation}
\widetilde{u}(x,p)=\frac{\cosh\left(\sqrt{\frac{p+\beta}{\alpha}}x\right)}{\sqrt{\frac{p+\beta}{\alpha}}\sinh\sqrt{\frac{p+\beta}{\alpha}}}.
\end{equation}
For example, if $\rho \equiv 1$ then we have $W(p)=-\frac{\alpha}{p+\beta}.$

\begin{theorem}
	Suppose for some $\nu>0$ such that $\beta < \nu < \beta + \pi^{2} \alpha$ the frequency-domain condition 
	\begin{enumerate}
		\label{EQ: FrequencyCondParab}
		\item[\textbf{(PF)}] $\operatorname{Re}W(i \omega - \nu) + 1/\mu_{0} > 0$ \text{ for all } $\omega \in [-\infty,+\infty]$
	\end{enumerate}
    is satisfied. Then there exists a compact self-adjoint operator $P^{*}=P \in \mathcal{L}(\mathbb{H},\mathbb{H})$ such that for some $\delta>0$ we have \textbf{(H1)}, \textbf{(H2)} with $j=1$ and \textbf{(H3)} satisfied for the cocycle $(\psi,\vartheta)$ given in Proposition \ref{PROP: ParabCocycle}.
\end{theorem}
\begin{proof}
	Frequency-domain condition \textbf{(PF)} along with the regularity of the linear problem and the exponential stabilizability of the pair $(A+\nu I, B)$ (see Subsection \ref{SUBSEC: LinearPartParab}) allows us to apply Theorem 4 of \cite{Likhtarnikov1976} to get a self-adjoint operator $P^{*}=P \in \mathcal{L}(\mathbb{H})$ such that 
	\begin{equation}
	\label{EQ: FreqFormIneqParab}
	2 ((A+\nu I)u + B\xi, Pu) + F(u,\xi)  \leq -\delta (|u|^{2}_{\mathbb{H}_{1}} + |\xi|^{2}), \text{ for } u \in \mathbb{H}_{1}, \xi \in \mathbb{R}.
	\end{equation}
	
	\noindent In addition, we have $P \in \mathcal{L}(\mathbb{H},\mathbb{H}_{1}) \cap \mathcal{L}(\mathbb{H}_{-1},\mathbb{H})$. Since the inclusion $\mathbb{H}_{1} \subset \mathbb{H}$ is compact, the operator $P \colon \mathbb{H} \to \mathbb{H}$ is compact.
	
	Putting $\xi = 0$ in \eqref{EQ: FreqFormIneqParab} and since the inclusion $\mathbb{H}_{1} \subset \mathbb{H}$ is bounded with the norm $1$, we get
	\begin{equation}
	\label{EQ: LyapunovInequalityParab}
	((A+\nu I)u,Pu) \leq -\delta |u|^{2}_{\mathbb{H}} \text{ for all } u \in \mathbb{H}_{1}.
	\end{equation}
	But for $u \in W^{2,2}(0,1) =: \mathcal{D}(A) \subset \mathbb{H}_{1}$ we have $Au \in \mathbb{H}$ and, consequently, for such $u$ the inequality in \eqref{EQ: LyapunovInequalityParab} is satisfied in the sense of $\mathbb{H}$. Thus from \eqref{EQ: LyapunovInequalityParab} and Proposition \ref{PROP: TrivialKernel} we get that $\operatorname{Ker}P=0$. In virtue of Propositions \ref{PROP: EstimatesSubspaces} and \ref{PROP: EstimateEigenvaluesWithStable} we also get that $\dim \mathbb{H}^{-}=1$.
	
	Putting $u=u_{1}-u_{2}$ and $\xi=f(Cu_{1})-f(Cu_{2})$ in \eqref{EQ: FreqFormIneqParab} and using property \textbf{(Q1)} of the form $F$ we get
	\begin{equation}
	\label{EQ: ParabFreqMonotonIneq}
	2(A(u_{1}-u_{2}) + B(f(Cu_{1})-f(Cu_{2})),P(u_{1}-u_{2})) \leq -\delta |u_{1}-u_{2}|^{2}.
	\end{equation}
	
	By the above properties of $P$ and \eqref{EQ: ParabFreqMonotonIneq}, for any two weak solutions $u_{j}(t)=u_{j}(t,t_{0},u_{j}(0))$, where $j=1,2$ and $t \geq t_{0}$, to \eqref{EQ: AbstractParabolicProb} we have for almost all $s \geq t_{0}$
	\begin{equation}
	\label{EQ: ParabFormDifferentiation}
	\begin{split}
	\frac{d}{ds}\left[ e^{2\nu s} V(u_{1}(s)-u_{2}(s)) \right] 
	\leq -\delta e^{2\nu s} |u_{1}(s)-u_{2}(s)|^{2}_{\mathbb{H}}.
	\end{split}
	\end{equation}
	Integrating \eqref{EQ: ParabFormDifferentiation} on any segment $[l,r]$, where $l \geq t_{0}$, we get \textbf{(H3)}.
\end{proof}

Summarizing the above, we obtain the following theorem.
\begin{theorem}
	\label{TH: FinalParabBoundaryTh}
	Let $(\psi,\vartheta)$ be the $\sigma$-periodic cocycle generated by \eqref{EQ: AbstractParabolicProb} and suppose \textbf{(PF)} is satisfied with some $\nu>0$ such that $\beta < \nu < \beta + \pi^{2} \alpha$. Then any bounded in the future trajectory converges to a $\sigma$-periodic trajectory.
\end{theorem}

In the special case $\rho \equiv 1$ the inequality in \textbf{(PF)} is equivalent to
\begin{equation}
-\frac{\alpha(\beta - \nu)}{(\beta-\nu)^{2} + \omega^{2}} + 1/\mu_{0} > 0
\end{equation}
Obviously, we have it always satisfied if $\nu > \beta$.  However, if it holds with $\nu=0$, i.~e. we have $\mu_{0} < \frac{\beta}{\alpha}$, then there is a unique stationary solution for the unperturbed problem, i.~e. with $g \equiv 0$ and $f$ independent of $t$. If $\mu_{0} \geq \frac{\beta}{\alpha}$ then there may be nontrivial stationary solutions and, consequently, under the conditions of Theorem \ref{TH: FinalParabBoundaryTh} there may be several $\sigma$-periodic trajectories.
\section*{Acknowledgements}

I thank my supervisor V.~Reitmann for many useful discussions on the topic and D.~N.~Cheban for his encouragement.

\appendix
\section{Frequency theorem of the Yakubovich-Likhtarnikov for $C_{0}$-semigroups}
\label{SEC: FrequencyTheorem}
In this section $\mathbb{H}$ and $\Xi$ denote Hilbert spaces over $\mathbb{C}$. Let $A$ be the generator of a $C_{0}$-semigroup $G(t)$, $t \geq 0$, in $\mathbb{H}$ with the domain $\mathcal{D}(A)$. Suppose $B \in \mathcal{L}(\Xi, \mathbb{H})$ is given. The equation
\begin{equation}
\label{EQ: ControlSystemInH}
\dot{u} = Au + B\xi
\end{equation}
is called a \textit{control system}. It is well know that for $T>0$ and every \textit{control function} $\xi \in L^{2}(0,T;\mathcal{U})$ and $u_{0} \in \mathbb{H}$ there exists a unique mild solution $u(t)=u(t,u_{0},\xi)$, $t \in [0,T]$, to \eqref{EQ: ControlSystemInH} satisfying $u(0)=u_{0}$.

For any bounded operator $C \in \mathcal{L}({\mathbb{H}, \Xi})$ the operator $A+BC$ will be the generator of some $C_{0}$-semigroup $G_{C}(t)$, $t \geq 0$ (see Theorem 7.5 in \cite{Krein1971}). The pair $(A,B)$ is called 
\begin{enumerate}
	\item $L_{2}$-\textit{stabilizable} if $C \in \mathcal{L}(\mathbb{H},\Xi)$ exists such that $G_{C}(t)u_{0} \in L_{2}([0,+\infty),\mathbb{H})$ for every $x_{0} \in \mathbb{H}$;
	\item \textit{exponentially stabilizable} if $C \in \mathcal{L}(\mathbb{H},\Xi)$ exists such that for some constants $M>0$ and $\varepsilon>0$ we have $\| G_{C}(t) \| \leq M e^{-\varepsilon t}$;
	\item $L_{2}$-\textit{controllable} if for every $u_{0} \in \mathbb{H}$ there exists a control $\xi \in L_{2}(0,+\infty; \Xi)$ such that $u(\cdot,u_{0},\xi) \in L_{2}(0,+\infty; \mathbb{H})$.
\end{enumerate}

It is not hard to see that an exponential stabilizable pair is $L_{2}$-stabilizable and an $L_{2}$-stabilizable pair is $L_{2}$ controllable. However, it turns out that these properties are equivalent (this is a byproduct of the frequency theorem, see \cite{LouisWexler1991, Likhtarnikov1977}).

Suppose we are given with a continuous Hermitian form on $\mathbb{H} \times \Xi$:
\begin{equation}
F(u,\xi):=(F_{1}u,u) + 2\operatorname{Re}(F_{2}u,\xi) + (F_{3}\xi,\xi),
\end{equation}
where $F^{*}_{1}=F_{1} \in \mathcal{L}(\mathbb{H},\mathbb{H})$, $F_{2} \in \mathcal{L}(\mathbb{H},\Xi)$ and $F^{*}_{3}=F_{3} \in \mathcal{L}(\Xi,\Xi)$. We introduce the number

\begin{equation}
\alpha_{2} :=\sup \frac{F(u,\xi)}{|u|^{2}+|\xi|^{2}},
\end{equation}
where the supremum is taken over all triples $(\omega,u,\xi) \in \mathbb{R} \times \mathbb{H} \times \Xi$ such that $i\omega(u,v) = (u,A^{*}v) + (B\xi,v)$ holds for all $v \in \mathcal{D}(A^{*})$. If for some $\varepsilon>0$ every point from the strip $|\operatorname{Re} \lambda| \leq \varepsilon$ is regular for $A$, we consider also the value
\begin{equation}
\alpha_{3}:=\sup\limits_{\omega \in \mathbb{R}}\sup\limits_{\xi \in \Xi} \frac{F((i\omega I - A)^{-1}B\xi,\xi)}{|\xi|^{2}}.
\end{equation}

The following theorem is originally proved by Yakubovich and Likhtarnikov in \cite{Likhtarnikov1977} (see Theorem 3 therein). The main result of \cite{Likhtarnikov1977} was rediscovered later by Louis and Wexler in \cite{LouisWexler1991} (see Theorem 2 therein).
\begin{theorem}
	\label{TH: YakubovichLikhtarnikovC0}
	Suppose $(A,B)$ is $L_{2}$-controllable. Then the following properties are equivalent 
	\begin{enumerate}
		\item There exists a self-adjoint operator $P^{*}=P \in \mathcal{L}(\mathbb{H},\mathbb{H})$ such that for some $\delta>0$ the inequality
		\begin{equation}
		2\operatorname{Re}(Au+B\xi, Pu) + F(u,\xi) \leq -\delta \left(|u|^{2} + |\xi|^{2} \right)
		\end{equation}
		is satisfied for all $u \in \mathcal{D}(A)$ and $\xi \in \Xi$.
		
		\item $\alpha_{2}<0$.
		
		Moreover, if $\alpha_{3}$ is well-defined in the above sense then (2) is equivalent to $\alpha_{3}<0$.
	\end{enumerate}
\end{theorem}

\begin{remark}
	In practise control systems are usually considered in real Hilbert spaces and Theorem \ref{TH: YakubovichLikhtarnikovC0} is applied to the complexifications of these spaces and corresponding linear operators. Thus, in this case the operator $P=P^{*}$ from Theorem \ref{TH: YakubovichLikhtarnikovC0} acts in the complexification\footnote{Recall that the \textit{complexification} of a real Hilbert space $\mathbb{H}$ with the scalar product $(\cdot,\cdot)_{\mathbb{H}}$ is the external direct sum $\mathbb{H}^{\mathbb{C}} := \mathbb{H} \oplus \mathbb{H}$ endowed with the multiplication $(a+ i b) (u,v) := (au - bv, av + bu)$ for $a,b \in \mathbb{R}$ and $u,v \in \mathbb{H}$ and the inner product
		\begin{equation*}
		\langle (u_{1},v_{1}), (u_{2},v_{2}) \rangle := (u_{1}, u_{2})_{\mathbb{H}} - i(u_{1},v_{2})_{\mathbb{H}} + i (v_{1},u_{2})_{\mathbb{H}} + (v_{1},v_{2})_{\mathbb{H}}.
		\end{equation*}
		For a bounded linear operator $A \colon \mathbb{H} \to \mathbb{H}$ the complexification $A^{\mathbb{C}} \colon \mathbb{H}^{\mathbb{C}} \to \mathbb{H}^{\mathbb{C}}$ is defined by $A^{\mathbb{C}}(u,v):=(Au,Av)$, where $u,v \in \mathbb{H}$. For every quadratic form $F$ in $\mathbb{H}$ there is corresponding \textit{Hermitian} extension $F^{\mathbb{C}}$ of $F$ defined by $F^{\mathbb{C}}(u,v):=F(u)+F(v)$ for $u,v \in \mathbb{H}$.
	} 
	$\mathbb{H}^{\mathbb{C}}$ of some real Hilbert space $\mathbb{H}$. Note that any operator $P \colon \mathbb{H}^{\mathbb{C}} \to \mathbb{H}^{\mathbb{C}}$ can be represented by a $2\times 2$-matrix
	\begin{equation}
	\label{EQ: MatrixRepresentation}	
	\begin{bmatrix}
	\widetilde{P} & Q\\
	-Q & \widetilde{P}
	\end{bmatrix},
	\end{equation}
	where $\widetilde{P}, Q \in \mathcal{L}(\mathbb{H})$. It can be shown that the equality $P=P^{*}$ is equivalent to the following two conditions: $\widetilde{P}^{*} = \widetilde{P}$ and $(Qu,u)=0$, for all $u \in \mathbb{H}$. If the control problem \eqref{EQ: ControlSystemInH} is posed in real Hilbert spaces $\mathbb{H}$ and $\Xi$ with the real quadratic form $F$, from Theorem \ref{TH: YakubovichLikhtarnikovC0} applied to complexifications $A^{\mathbb{C}}, B^{\mathbb{C}}$ and the Hermitian extension $F^{\mathbb{C}}$ of $F$ we get the operator $P=P^{*} \colon \mathbb{H}^{\mathbb{C}} \to \mathbb{H}^{\mathbb{C}}$. It is easy to check that the self-adjoint operator $\widetilde{P} \colon \mathbb{H} \to \mathbb{H}$ defined in \eqref{EQ: MatrixRepresentation} satisfy for all $u \in \mathcal{D}(A)$ and $\xi \in \Xi$ the inequality
	\begin{equation}
	2(Au+B\xi,\widetilde{P} u) + F(u,\xi) \leq -\delta ( |u|^{2} + |\xi|^{2} ).
	\end{equation}
	It is clear that $\widetilde{P}$ inherits the compactness of $P$ if it holds.
\end{remark}
\section{Spectral properties of solutions to Lyapunov inequalities}
\label{SEC: LyapunivIneq}

In this section $A \colon \mathcal{D}(A) \subset \mathbb{H} \to \mathbb{H}$ is the generator of a $C_{0}$-semigroup $G(t)$, $t \geq 0$, in a real Hilbert space $\mathbb{H}$ and $P$ is a bounded self-adjoint operator in $\mathbb{H}$. Using the functional calculus of self-adjoint operators (\cite{Helemskii2006}, p. 370--371) we get the decomposition of $\mathbb{H}$ into three orthogonal $P$-invariant subspaces as $\mathbb{H} = \mathbb{H}^{+}$ $\oplus$ $\mathbb{H}^{-} \oplus \mathbb{H}^{0}$, where $\restr{P}{\mathbb{H}^{+}}>0$, $\restr{P}{\mathbb{H}^{-}}< 0$ and $\restr{P}{\mathbb{H}^{0}}=0$, i.~e. $\mathbb{H}^{0}=\operatorname{Ker}P$.

\begin{proposition}
	\label{PROP: TrivialKernel}
	Suppose that for some $\delta>0$ the inequality
	\begin{equation}
	\label{EQ: LyapunovIneqProp}
	(Au,Pu) \leq -\delta |u|^{2}
	\end{equation}
	holds for all $u \in \mathcal{D}(A)$. If $\mathcal{R}(A)=\mathbb{H}$ (in particular, if $0 \notin \sigma(A)$) then $\operatorname{Ker}(P)=\{ 0 \}$.
\end{proposition}
\begin{proof}
	Indeed, let $v \in \operatorname{Ker}(P)$. Then there is $u \in \mathcal{D}(A)$ such that $Au=v$. From \eqref{EQ: LyapunovIneqProp} we get
	\begin{equation}
	(Au,Pu) = (PAu,u) = 0 \leq -\delta |u|^{2}
	\end{equation}
	and, consequently, $u=0$ and $Au=v=0$.
\end{proof}

Inequality \eqref{EQ: LyapunovIneqProp} is called \textit{Lyapunov inequality}. In the applications given in Sections \ref{SEC: DelayedSystems} and \ref{SEC: Parabolic} the operator $P$ will be obtained as a solution to certain operator inequalities that in particular include \eqref{EQ: LyapunovIneqProp}. Proposition \ref{PROP: TrivialKernel} gives a simple criteria for the space $\mathbb{H}^{0}$ to be zero dimensional. In order to study the dimension of $\mathbb{H}^{-}$ the following simple lemmas are useful.

\begin{proposition}
	\label{PROP: EstimatesSubspaces}
	Suppose $\mathbb{H} = \mathbb{H}^{s} \oplus \mathbb{H}^{u}$ with $\dim \mathbb{H}^{u}=j < \infty$. We have the following 
	\begin{enumerate}
		\item[1)] If $\restr{P}{\mathbb{H}^{s}} \geq 0$ then $\dim \mathbb{H}^{-} \leq j$.
		\item[2)] If $\restr{P}{\mathbb{H}^{u}} < 0$ then $\dim \mathbb{H}^{-} \geq j$.
 	\end{enumerate}
\end{proposition}
\begin{proof}
	1) Suppose $k > j$ and $v_{1},\ldots,v_{k} \in \mathbb{H}^{-}$. For every $l=1,\ldots,k$ there exists a unique representation
	\begin{equation}
	\label{EQ: RepresentationLemma}
	v_{l}=v^{s}_{l}+v^{u}_{l}
	\end{equation}
	with $v^{s}_{l} \in \mathbb{H}^{s}$ and $v^{u}_{l} \in \mathbb{H}^{u}$. Since $k > j = \dim \mathbb{H}^{u}$, the vectors $v^{u}_{1},\ldots, v^{u}_{k}$ are linearly dependent and, consequently, we have 
	\begin{equation}
	\sum_{l=1}^{k}c_{l} v^{u}_{l} = 0
	\end{equation}
	for some numbers $c_{1},\ldots, c_{l}$. From \eqref{EQ: RepresentationLemma} we have
	\begin{equation}
	\sum_{l=1}^{k}c_{l} v_{l} = \sum_{l=1}^{k}c_{l} v^{s}_{l} + \sum_{l=1}^{k}c_{l} v^{u}_{l} = \sum_{l=1}^{k}c_{l} v^{s}_{l} = 0,
	\end{equation}
	where the last equality is due to $\mathbb{H}^{-} \cap \mathbb{H}^{s} = \{ 0 \}$. Therefore every $k > j$ vectors from $\mathbb{H}^{-}$ are linearly dependent. So, $\dim \mathbb{H}^{-} \leq j$.
	
	2) Suppose that $\dim \mathbb{H}^{-} < j$. Then there exists a non-zero vector $v \in \left(\mathbb{H}^{-}\right)^{\bot} \cap \mathbb{H}^{u}$. So $(Pv,v) < 0$ and $(Pv,v) \geq 0$ at the same time. This is a contradiction.
\end{proof}

\begin{proposition}
	\label{PROP: EstimateEigenvaluesWithStable}
	Suppose \eqref{EQ: LyapunovIneqProp} is satisfied and $\mathbb{H}$ splits into the direct sum $\mathbb{H} = \mathbb{H}^{s} \oplus \mathbb{H}^{u}$ of $G(t)$-invariant subspaces $\mathbb{H}^{s}$ and $\mathbb{H}^{u}$ with $\dim \mathbb{H}^{u} = j < \infty$. Then
	\begin{enumerate}
		\item[1)] If $G(t)u_{0} \to 0$ as $t \to +\infty$ for $u_{0} \in \mathbb{H}^{s}$ then $\restr{P}{\mathbb{H}^{s}} > 0$.
		\item[2)] If $G(t)$ is invertible on $\mathbb{H}^{u}$ and  $G(t)u_{0} \to 0$ as $t \to -\infty$ for $u_{0} \in \mathbb{H}^{u}$ then $\restr{P}{\mathbb{H}^{u}} < 0$.
	\end{enumerate}
\end{proposition}
\begin{proof}
	1) Put $V(u):=(Pu,u)$. Differentiating\footnote{Since $A$ is the generator of the $C_{0}$-semigroup $G(t)$, the function $t \mapsto G(t)u_{0}$ is continuously differentiable provided $u_{0} \in \mathcal{D}(A)$ \cite{Krein1971}.} $V(G(t)u_{0})$ w. r. t. $t$ for $u_{0} \in \mathcal{D}(A)$, using \eqref{EQ: LyapunovIneqProp}, then integrating it on $[0,t]$ and extending the inequality for all $u_{0} \in \mathbb{H}$ by the continuity we have
	\begin{equation}
	\label{EQ: LyapunovInequality}
	V(G(t)u_{0}) - V(u_{0}) \leq -\delta \int_{0}^{t}|G(s)u_{0}|^{2}ds.
	\end{equation}
	Taking it to the limit in \eqref{EQ: LyapunovInequality} as $t \to \infty$ for $u_{0} \in \mathbb{H}^{s}$ we have
	\begin{equation}
	V(u_{0}) \geq \delta \int_{0}^{\infty} |G(t)u_{0}|^{2}ds.
	\end{equation}
	2) Analogously to $1)$ for $u_{0} \in \mathbb{H}^{u}$ one may deduce
	\begin{equation}
	V(u_{0}) \leq -\delta \int_{-\infty}^{0} |G(t)u_{0}|^{2}ds.
	\end{equation}
\end{proof}




\begin{thebibliography}{00}

\bibitem{Anik2020PB}
Anikushin M.~M. The Poincar\'{e}-Bendixson theory for certain semi-flows in Hilbert spaces, \textit{arXiv preprint} arXiv:2001.08627 (2020).

\bibitem{Anikushin2019+OnCom} Anikushin M. M. On the compactness of solutions to certain operator inequalities arising from the Likhtarnikov-Yakubovich frequency theorem (2020, to appear)

\bibitem{Anikushin2019Vestnik}
Anikushin M. M. On the Liouville phenomenon in estimates of fractal dimensions of forced quasi-periodic oscillations, \textit{Vestnik St. Petersb. Univ. Math.}, \textbf{52}(3), 234--243 (2019).

\bibitem{AnikushinRR2019}
Anikushin M. M., Reitmann V., Romanov A. O. Analytical and numerical estimates of the fractal dimension
of forced quasiperiodic oscillations in control systems, \textit{Differential Equations and Control Processes} (Differencialnie Uravnenia i Protsesy Upravlenia), \textbf{87}(2) (2019).

\bibitem{Anikushin2019ND}
Anikushin M. M. On the Smith reduction theorem for almost periodic ODEs satisfying the squeezing property, \textit{Rus. J. Nonlin. Dyn.}, \textbf{15}(1), 97--108 (2019).

\bibitem{ArovYakub1982}
Arov D. Z. Yakubovich V. A. Semiboundedness conditions for quadratic functionals on Hardy spaces, \textit{Vestn. Leningr. Univ., Ser. Mat. Mekh. Astron.}, 1, 7--13 (1982).

\bibitem{BatkaiPiazzera2005}
B\'{a}tkai A., and Piazzera S. \textit{Semigroups for Delay Equations}. A K Peters, Wellesley (2005).

\bibitem{Cartwright1969}
Cartwright M. L. Almost periodic differential equations and almost periodic flows, \textit{J. Differ. Equations}, \textbf{5}, 167–181 (1969).

\bibitem{CarvalhoLangaRobinson2012}
Carvalho A., Langa J. A., Robinson J. \textit{Attractors for Infinite-Dimensional Non-Autonomous Dynamical Systems}. Springer Science \& Business Media (2012).

\bibitem{Cheban2020Book}
Cheban D. N. \textit{Nonautonomous Dynamics: Nonlinear oscillations and Global attractors}, Springer Nature (2019).

\bibitem{Cheban2019}
Cheban D., Liu Z. Poisson stable motions of monotone nonautonomous dynamical systems, \textit{Science China Mathematics}, \textbf{62}(7), 1391–1418 (2019).

\bibitem{Faheem1987} Faheem M., Rao M. R. M. Functional differential equations of delay type and nonlinear evolution in $L_{p}$-spaces, \textit{J. Math. Anal. and Appl.}, \textbf{123}(1), 73--103 (1987).

\bibitem{FengWangWu2017}
Feng L., Wang Yi and Wu J. Semiflows “Monotone with Respect to High-Rank Cones" on a Banach Space. \textit{SIAM J. Math. Anal.}, \textbf{49}(1), 142--161 (2017).

\bibitem{Gelig1978}
Gelig A. Kh., Leonov G. A., Yakubovich V. A. \textit{Stability of Nonlinear Systems with Non-Unique Equilibrium State}. Nauka, Moscow (1978).

\bibitem{Hale}
Hale J. K., Lunel S. M. V. \textit{Introduction to Functional Differential Equations}. Springer Science \& Business Media (1993).

\bibitem{Helemskii2006}
Helemskii A. Ya. \textit{Lectures and Exercises on Functional Analysis}. Providence, RI: American Mathematical Society, 2006.

\bibitem{KalininReitmann2012}
Kalinin Yu. N., Reitmann V. Almost periodic solutions in control systems with monotone nonlinearities. \textit{Differential Equations and Control Processes} (Differencialnie Uravnenia i Protsesy Upravlenia), \textbf{61}(4) (2012).

\bibitem{Krein1971}
Krein S. G. \textit{Linear Differential Equations in Banach Space}, AMS, (1971).

\bibitem{KuzLeoReit2019}
Kuznetsov N. V., Leonov G. A., Reitmann V. \textit{Attractor Dimension Estimates for Dynamical Systems: Theory and Computation}. Switzerland: Springer International Publishing AG (2020).

\bibitem{Lasiecka1983}
Lasiecka I., Triggiani R. Stabilization of Neumann boundary feedback of parabolic equations: The case of trace in the feedback loop, \textit{Applied Mathematics and Optimization}, \textbf{10}(1), 307--350 (1983).

\bibitem{LeoBurShep1996}
Leonov G. A., Burkin I. M., Shepeljavyi A. I. \textit{Frequency Methods in Oscillation Theory}. Kluwer Academic Publishers (1996).

\bibitem{LevitanZhikov1982}
Levitan B. M., Zhikov V. V. \textit{Almost Periodic Functions and Differential Equations}. CUP Archive (1982).

\bibitem{Likhtarnikov1977}
Likhtarnikov A. L., Yakubovich V. A. The frequency theorem for continuous one-parameter semigroups. \textit{Math. USSR-Izv.} \textbf{41}(4), 895--911 (1977) [in Russian].

\bibitem{Likhtarnikov1976}
Likhtarnikov A. L., Yakubovich V. A. The frequency theorem for equations of evolutionary type, \textit{Sib. Math. J.}, \textbf{17}(5), 790--803 (1976).

\bibitem{Lions1971}
Lions J. L. \textit{Optimal Control of Systems Governed by Partial Differential Equations}, Springer-Verlag (1971).

\bibitem{Liu2003}
Liu W. Boundary feedback stabilization of an unstable heat equation, \textit{SIAM journal on control and optimization}. \textbf{42}(3), 1033--1043 (2003).

\bibitem{LouisWexler1991}
Louis J.-Cl., Wexler D. The Hilbert space regulator problem and operator Riccati equation under stabilizability, \textit{Annales de la Soci\'{e}t\'{e} Scientifique de Bruxelles}, \textbf{105}(4), 137--165 (1991).

\bibitem{Marcus1989}
Marcus C. M., Westervelt R. M. Stability of analog neural networks with delay, \textit{Phys. Rev. A}, \textbf{39}(1), 347--359 (1989).

\bibitem{Massera1950}
Massera J. L. The existence of periodic solutions of systems of differential equations. \textit{Duke Math. J.}, \textbf{17}(4), 457--475 (1950).

\bibitem{Pliss1966}
Pliss V. A. \textit{Nonlocal Problems of the Theory of Oscillations}, Academic Press, New York, (1966).

\bibitem{Popov2014}
Popov S., Reitmann V. Frequency domain conditions for finite-dimensional projectors and determining observations for the set of amenable solutions. \textit{Discrete \& Continuous Dynamical Systems-A.}, \textbf{34}(1), 249--267 (2014).

\bibitem{Smith1994}
Smith R. A. Orbital stability and inertial manifolds for certain reaction diffusion systems. \textit{P. Lond. Math. Soc.}, \textbf{3}(1), 91--120 (1994).

\bibitem{Smith1990}
Smith R. A. Convergence theorems for periodic retarded functional differential equations, \textit{P. Lond. Math. Soc.}, \textbf{3}(3), 581--608 (1990).

\bibitem{Smith1986}
Smith R. A. Massera's convergence theorem for periodic nonlinear differential equations, \textit{J. Math. Anal. and Appl.}, \textbf{120}(2), 679--708 (1986).

\bibitem{Webb1976}
Webb G. F. Functional differential equations and nonlinear semigroups in $L^{p}$-spaces. \textit{J. Differ. Equations}, \textbf{20}(1), 71--89 (1976).

\bibitem{Webb1981}
Webb G. F., Badii M. Nonlinear nonautonomous functional differential equations in $L^{p}$ spaces, \textit{Nonlinear Anal-Theor.}, \textbf{5}(2), 203--223 (1981).

\end{thebibliography}


\end{document}